\documentclass[10pt,twoside,titlepage]{article}
\usepackage{amsmath}
\usepackage{latexsym,amssymb}

\title{}
\author{}
\newcommand{\lastname}{}
\makeatletter

\newcommand{\rightheadline}{\hfill\sc\lastname\hfill}
\newcommand{\leftheadline}{\hfill\sc\lastname\hfill}
\newcommand{\firstheadline}{}

\def\ps@headings{\let\@mkboth\@gobbletwo
    \gdef\@oddhead{\ifnum\value{page}=1\firstheadline
                  \else\rightheadline\rm\thepage\fi}
    \gdef\@oddfoot{}
    \gdef\@evenhead{\ifnum\value{page}=1\firstheadline
                  \else\rm\thepage\leftheadline\fi}
    \gdef\@evenfoot{}
}

\def\section{\@startsection{section}{1}{\z@}{-\bigskipamount}{\medskipamount}
{\centering\bf}}

\countdef\revised=100

\newcommand{\qed}{{\unskip\nobreak\hfil\penalty50\hskip .001pt \hbox{}
          \nobreak\hfil
          \vrule height 1.2ex width 1.1ex depth -.1ex
           \parfillskip=0pt\finalhyphendemerits=0\medbreak}\rm}

\renewcommand{\lastname}{Reilly}

\def\trivlist{\vspace{-\lastskip}\medbreak
 \@trivlist \labelwidth\z@ \leftmargin\z@
\itemindent\z@   \let\@itemlabel\@empty%
\def\makelabel##1{##1}}

\def\@thmcounter#1{\noexpand\arabic{#1}}
\def\@thmcountersep{}
\def\@begintheorem#1#2{\trivlist \item[\hskip 
\labelsep{\bf #1\ #2.\quad}]\it}
\def\@opargbegintheorem#1#2#3{\it \trivlist
      \item[\hskip \labelsep{\bf #1\ #2.\quad{\rm #3}}]}

\newtheorem{theorem}{Theorem}[section]
\newtheorem{proposition}[theorem]{Proposition}
\newtheorem{lemma}[theorem]{Lemma}
\newtheorem{corollary}[theorem]{Corollary}

\newtheorem{notation}[theorem]{Notation}

\newenvironment{proof}{\begin{trivlist}\item[\hskip%
\labelsep{\bf Proof.\quad}]}%
{\hfill\qed\rm\end{trivlist}}

\def\Proof.{\rm
  \par\ifdim\lastskip<\medskipamount\vspace{-\lastskip}\fi\smallskip
           \noindent {\bf Proof.}\enspace}

\def\tr{{\rm tr}\,}
\def\ltr{{\rm ltr}\,}
\def\rtr{{\rm rtr}\,}

\newcommand{\arttype}{}

\def\titlepage{\thispagestyle{headings}
 \vbox{\vspace{.5truecm}}
\noindent \arttype\par
\begin{center}
{\vbox{\baselineskip=15pt\large\bf \@title\boldmath}}
  \vskip .75truecm  
{\bf \@author} 
\end{center}

\def\abstract{\baselineskip=8pt\centerline{\bf Abstract}
\vspace{.05in} \noindent\quotation}
\def\endabstract{\if@twocolumn\else\endquotation\fi}}

\begin{document}

\input prepictex
\input pictex
\input postpictex

\ \\[0.5cm]

\title{Kernel Classes of Varieties of Completely Regular Semigroups I}
\author{Norman R. Reilly}
\titlepage

\vskip1.1cm

{\footnotesize
\begin{abstract}
\noindent
Several complete congruences on the lattice $\mathcal{L}(\mathcal{CR})$ of varieties of completely regular 
semigroups have been fundamental to studies of the structure of $\mathcal{L}(\mathcal{CR})$.  These are 
the kernel relation $K$, the left trace relation $T_{\ell}$, the right trace relation $T_r$ and their 
intersections $K\cap T_{\ell}, K\cap T_r$.  However, with 
the exception of the lattice of all band varieties which happens to coincide with the kernel class of the trivial variety, 
almost nothing is known about the internal structure of individual $K$-classes beyond the fact that they are 
intervals in $\mathcal{L}(\mathcal{CR})$.  Here we present a number of general results that are 
pertinent to the study of $K$-classes.  This includes a variation of the renowned Pol\'{a}k Theorem and 
its relationship to the complete retraction $\mathcal{V} \longrightarrow \mathcal{V}\cap \mathcal{B}$, 
where $\mathcal{B}$ denotes the variety of bands.   These results are then applied, here and in a sequel, 
to the 
detailed analysis of certain families of $K$-classes.  The paper concludes with results hinting at the 
complexity of $K$-classes in general, such as that the classes of relation $K/K_{\ell}$ may have the 
cardinality of the continuum. 
\end{abstract}
}
\vskip1cm

\section{Introduction}

After presenting the necessary background information, we will develop an alternative formulation of the important 
representation of the interval $[\mathcal{S}, \mathcal{CR}]$ due to Pol\'{a}k.  Here 
$\mathcal{S, CR}$ denote the varieties of semilattices and completely regular semigroups, respectively.  The 
objectives are twofold.  Pol\'{a}k's Theorem contains certain key conditions that are phrased in 
the language of fully invariant congruences on the free completely regular semigroup on a countably infinite set.  
However, in many situations 
it is more intuitive to work in the context of varieties.  Of course these are equivalent, but it is often simpler to work 
in the language of varieties.   So the first goal is to develop a formulation of Pol\'{a}k's result 
completely in terms of varieties and operators on the lattice of varieties.  In doing so we obtain a second goal which is 
a series of conditions that are easier to work with in many situations.  Trotter [T] showed that the mapping $\mu_{\mathcal{B}}: 
\mathcal{V} \longrightarrow \mathcal{V}\cap \mathcal{B}$, where $\mathcal{B}$ is the variety of bands, is a complete 
retraction of  $\mathcal{L}(\mathcal{CR})$ onto $\mathcal{L}(\mathcal{B})$, thereby inducing a complete congruence on $\mathcal{L}(\mathcal{CR})$ which, following Petrich,  
we denote by ${\bf B}$.   The classes of ${\bf B}$ are therefore intervals and we 
denote the ${\bf B}$-class of $\mathcal{V} \in \mathcal{L}(\mathcal{CR})$ by $[\mathcal{V}_{{\bf B}}, \mathcal{V}^{{\bf B}}]$.   
This provides another framework  
with respect to which we can analyze certain $K$-classes and leads us to demonstrate how  the description of 
varieties of the form $\mathcal{V}^{\bf B}$ due to Reilly and Zhang [RZ2000] is an integral part of 
Pol\'{a}k's Theorem.   In the final 
section, we study more specific representations of the intervals $[\mathcal{S}, \mathcal{O}]$ and 
$[\mathcal{S}, L\mathcal{O}]$, where $\mathcal{O}$ (respectively, $L\mathcal{O}$) is the variety of all orthodox 
(respectively, locally orthodox) completely regular semigroups associated 
with  Pol\'{a}k's Theorem and, in each case, 
break them down into the natural partition corresponding to whether a variety contains the variety 
$\mathcal{B}$ or not.\\

Apart from a relatively short
list of varieties (orthodox, locally orthodox, bands of groups and
some related varieties) not a great deal of information exists dealing
with the structure of the interval $[\mathcal{B}, \mathcal{CR}].$
In a sequel we will provide a detailed description of a large class of intervals that 
are contained within  $[\mathcal{B}, \mathcal{CR}]$ and within $K$-classes.  We will also 
examine in detail the structure of $K$-classes of the form 
$\mathcal{A}_p K$, where $\mathcal{A}_p$ denotes the variety of abelian groups of prime 
exponent $p$.   We will completely determine the structure of the parts of these kernel 
classes that contain $\mathcal{B}$ and also the complement of those parts in  
$\mathcal{A}_p K$. 
\vskip1cm


\section{Background}

We refer the reader to the book [PR99] for a general background on
completely regular semigroups and for all undefined notation and
terminology.  All semigroups under consideration are assumed to be completely 
regular semigroups.  Unless otherwise mentioned, throughout this section $S$ denotes 
an arbitrary completely regular semigroup.

For $\rho$ an equivalence relation on a nonempty set $X$ and $x\in
X,$ $x\rho$ denotes the $\rho$-class  of $x.$ In any lattice $L,$
for $a,b \in L$ such that $a\leq b$, we define the \emph{interval} $[a,b]$
to be $\{c\in L \mid a\leq c\leq b\}.$

We denote by $E(S)$ the set of idempotents of $S$ 
and by $\mathcal{C}(S)$ the lattice of congruences on $S.$  For any equivalence relation $\rho$ 
on $S$, $\rho^0$ denotes the largest congruence on $S$ 
contained in the relation $\rho$.  We denote by $\tau_S$ (or simply $\tau$) the largest
idempotent pure congruence on $S$, that is, the largest congruence contained 
in the equivalence relation on $S$ with just the two equivalence classes $E(S)$ and $S\backslash E(S)$. 
For $\rho
\in \mathcal{C}(S),$ the \emph{kernel\/} and the \emph{trace\/} of
$\rho$ are given by $\ker \rho = \{a\in S \mid a\;\rho\;e \;\;\mbox{for
some}\;\; e\in E(S)\},$ $\,\tr\,\rho = \rho |_{E(S)}.$  The 
\emph{left\/} and the \emph{right traces\/} of $\rho $ are $\ltr \rho
= \tr (\rho \vee \mathcal{L})^0$ and $\rtr \rho = \tr (\rho \vee
\mathcal{R})^0,$ where the join is taken within the lattice of
equivalence relations on $S.$  These relations lead to several
important relations on the lattice of congruences $C(S)$ on $S$ defined as follows:
$$
\begin{array}{l}
\lambda K\rho \,\Longleftrightarrow\, \ker \lambda = \ker \rho \,,
\quad \lambda T_{\ell}\rho  \,\Longleftrightarrow\, \ltr\lambda  =
\ltr \rho \,, \\[0.2cm]
\lambda T\rho \,\Longleftrightarrow\, \tr \lambda
= \tr \rho \,, \quad \lambda T_r \rho \,\Longleftrightarrow\, \rtr
\lambda = \rtr \rho \,, \\[0.2cm]
K_{\ell} = K\cap T_{\ell}\,, \quad K_r = K\cap T_r\,.
\end{array}
$$
\vskip0.1cm

\noindent
The classes of these relations are intervals and we write
$$
\rho P = [\rho_P, \rho^P]\,, \qquad P\in \{K, T_{\ell}, T_r, 
K_{\ell}, K_r\}\,.
$$

As $S$ is a union of its
(maximal) subgroups, on $S$ we have a unary operation
$a\longrightarrow a^{-1},$ where $a^{-1}$ is the inverse of $a$ in
the maximal subgroup of $S$ containing $a.$ Hence for the purpose
of studying varieties of completely regular semigroups, it is customary 
to consider them as algebras with the binary operation of multiplication and the unary
operation of inversion. We write $a^0 = aa^{-1} \;(= a^{-1}a)$ for
any element $a$ of $S.$

The class $\mathcal{CR}$ of all completely regular semigroups
constitutes a variety. It is defined, within the class of unary
semigroups by the identities
$$
a(bc) = (ab)c\,, \quad a = aa^{-1}a\,, \quad (a^{-1})^{-1} = a\,,
\quad aa^{-1} = a^{-1}a\,.
$$

For $\mathcal{V} \in \mathcal{L}
(\mathcal{CR}),$ we denote by $\mathcal{L}(\mathcal{V})$  the lattice of
all subvarieties of $\mathcal{V}.$ 
The following varieties occur frequently:  $\mathcal{T, LZ, RZ, RB, G, R\mbox{e}G, CS}$ - the varieties of 
trivial semigroups,  left zero semigroups, right zero semigroups, rectangular bands,  groups, rectangular groups and completely simple semigroups, 
respectively, all of which are subvarieties of $\mathcal{CS}$, and $\mathcal{S, LNB, RNB, LRB, RRB, RRO, B, O}$ - 
the varieties of semilattices, left normal bands, 
right normal bands, left regular bands, right regular bands, right regular orthogroups, bands and 
orthogroups, respectively, 
all of which are varieties of completely regular semigroups containing 
$\mathcal{S}$.   \\  

\noindent
For any varieties $\mathcal{U}\in \mathcal{L}(\mathcal{G}), \mathcal{V}\in \mathcal{L}({CS}), \mathcal{W} 
\in \mathcal{L}(\mathcal{CR})$, we define 
\begin{eqnarray*}
H\mathcal{U} &=&  \{S\in \mathcal{CR}| H_e \in \mathcal{U}\;\mbox{for all}\: e \in E(S)\},\\
L\mathcal{W} &=& \{S\in \mathcal{CR}| eSe \in \mathcal{W}\;\mbox{for all}\: e \in E(S)\}\\
\mathcal{OW} &=& \mathcal{O} \cap \mathcal{W}.
\end{eqnarray*}
The classes $H\mathcal{U}, D\mathcal{V}$ and $L\mathcal{W}$ are all varieties.  
If $eSe \in \mathcal{W}$ for all $e \in E(S)$, then we say that  $S$ is {\em locally} (in) $\mathcal{W}$.  In particular, $L\mathcal{O}$ denotes the variety of {\em locally orthodox} completely regular semigroups 
and $L\mathcal{RRO}$ denotes the variety of locally right regular orthodox completely regular semigroups.  
See [PR99] for more information on all of the above varieties.\\

Let $X$ denote a countably infinite set and $CR_X$ denote the free completely regular semigroup on 
$X$.  See Clifford [C] for a description of $CR_X$.  The relations $K, T_{\ell}, T_r, K_{\ell}, K_r$ on 
the lattice of congruences on $CR_X$ restrict to complete congruences on the lattice of fully invariant congruences on $CR_X$ ([Pa], Theorems 9, 11).  \\

Via the usual antiisomorphism between 
the lattice of fully invariant congruences on the free completely regular semigroup $CR_X$ and 
$\mathcal{L}(\mathcal{CR})$, the relations $K,  
\,T_{\ell}, T_r, K_{\ell}, K_r$ defined above on $CR_X$ transfer to $\mathcal{L}
(\mathcal{CR})$ in a familiar way. We use the same notation for these
relations on $\mathcal{L}(\mathcal{CR})$ (and their intervals) as
for the corresponding ones on the lattice of congruences on $CR_X$.  In particular, for 
$P \in \{K, T_{\ell}, T_r, K_{\ell}, K_r\}, \mathcal{V} \in \mathcal{L}(\mathcal{CR})$, the $P$-class of 
$\mathcal{V}$ is an interval which we write as $[\mathcal{V}_P, \mathcal{V}^P]$.  This leads to two operators associated with each of 
these relations which are referred to as the {\em lower} and {\em upper} operators, respectively:
\begin{equation}
\mathcal{V}\; \longrightarrow \;\mathcal{V}_P \;\mbox{and}\; \mathcal{V} \longrightarrow \mathcal{V}^P.
\end{equation}
For these operators, we
write for example, $\mathcal{V}_K,$ $\mathcal{V}^{KT_r} =
(\mathcal{V}^K)^{T_r},$ and so on.  It is also useful to note that 
$$
\mathcal{U}\; P\; \mathcal{V} \Longleftrightarrow \mathcal{U}^P = \mathcal{V}^P 
\Longleftrightarrow \mathcal{U}_P = \mathcal{V}_P.
$$
Bases of identities for varieties of the form $\mathcal{V}^P$ can be obtained from 
any basis for $\mathcal{V}$ (see [J], [Pa], [PR90] and [R]).   


The following lemma lists just a few of the many known properties of these operators.

\begin{lemma}
Let $\mathcal{V} \in [\mathcal{S}, \mathcal{CR}].$
\begin{enumerate}
\item[{\rm (i)}] $(\mathcal{V}^K)_{T_{\ell}} = \mathcal{V}^K$.
\item[{\rm (ii)}]  The following intervals constitute the complete set of the $K$-classes 
within $\mathcal{L}(L\mathcal{O})$:
$$
\{[\mathcal{V}, \mathcal{O}H\mathcal{V}]\mid \mathcal{V} \in \mathcal{L}(\mathcal{G})\},\;
\{[\mathcal{V}, L\mathcal{O}D\mathcal{V}]\mid \mathcal{V} \in \mathcal{L}(\mathcal{CS}) 
\backslash \mathcal{L}(\mathcal{R}e\mathcal{G})\}
$$
In particular, $\mathcal{T}K = [\mathcal{T}, \mathcal{B}] = \mathcal{L}(\mathcal{B})$.
\item[{\rm (iii)}] The following intervals are $T_{\ell}$-classes:
$$
[\mathcal{S}, \mathcal{LRO}], [\mathcal{RNB}, L\mathcal{LRO}].
$$
\item[{\rm (iv)}] 
\begin{eqnarray*}
\mathcal{V}^{T_{\ell}} &=& \{S \in \mathcal{CR}| S/\mathcal{L}^0 \in \mathcal{V}\}\\
\mathcal{V}^{K_{\ell}} &=& \{S \in \mathcal{CR}| S/(\tau\cap \mathcal{L})^0 \in \mathcal{V}\}.
\end{eqnarray*}
\end{enumerate}
\end{lemma}
\begin{proof}
(i) See [Po2], Theorem 2.4(4).  
(ii) See [Po2], Theorem 2.6.
(iii)  See [Po2], Theorem 1.2 and Pastijn [Pa].  
(iv)  See [Pa] and [PR90].
\end{proof}

Let $\Theta $ be the set of all (nonempty) words over the alphabet
$\{T_{\ell}, T_r\}$ of the form $P_1\cdots P_n,$ where $P_i \in
\{T_{\ell}, T_r\}$ and $P_i \neq P_{i+1}$ for $i = 1,\ldots ,n-1$
with \vspace*{0.1cm}multiplication
$$
(P_1\cdots P_m) (Q_1\cdots Q_n) = \left\{ \begin{array}{ll}
P_1 \cdots P_mQ_1\cdots Q_n & \quad \mbox{if}\;\; P_m\neq Q_1\,,
\\[0.2cm]
P_1\cdots P_mQ_2\cdots Q_n & \quad \mbox{otherwise.}
\end{array} \right.
$$
\vskip0.1cm

\noindent
Clearly $\Theta $ is a semigroup. We adjoin the \emph{empty word}
$\emptyset$ to $\Theta $ to form $\Theta^1.$  (The monoid $\Theta^1$ was introduced 
in [R1990].)  We can consider $\Theta $ as 
being the free semigroup $\{T_{\ell},
T_r\}^+$ on the set $\{T_{\ell},T_r\}$ modulo the relations
$T_{\ell}^2 = T_{\ell},$ $T_r^2 = T_r.$ To emphasize that $\tau \in
\Theta $ is a  word in $T_{\ell}, T_r$ we might write $\tau = \tau
(T_{\ell}, T_r).$ If we replace every occurrence of $T_{\ell}$
(respectively, $T_r$) in $\tau $ by $K_{\ell}$ (respectively, $K_r$)
then we obtain a word over $\{K_{\ell}, K_r\}$ which we will denote
by $\tau (K_{\ell}, K_r).$  \\

Let $\tau\in \Theta^1$.  We carry over the notation for words in $X^+$ to $\Theta^1$ so that we  
denote by $|\tau|$ the {\em length} of $\tau$, with $|\emptyset| = 0$, and 
we denote by $h(\tau)$ (respectively, $t(\tau)$) the first (respectively, last) letter in $\tau$.  We also write $\overline{\tau}$ for the 
mirror image of the word $\tau$.\\

There is a natural way to associate a lower operator and an upper operator on the lattice $\mathcal{L}(\mathcal{CR})$ with any element $\tau \in \Theta$.  We simply define the operations by induction on the length of $\tau$.   Let $\tau\in \Theta^1$.  For $\tau = \emptyset$, we define $\mathcal{V}_{\emptyset} = 
\mathcal{V}^{\emptyset} = \mathcal{V}$.   If $|\tau| = 1$, then we have already defined 
such operators in (1) above.  So assume that 
$\tau = \sigma T_p$ where $p \in \{\ell, r\}$.  Let $\mathcal{V} \in \mathcal{L}(\mathcal{CR})$.  Then define
$$
\mathcal{V}_{\tau} = (\mathcal{V}_{\sigma})_{T_p},\;\; \mathcal{V}^{\tau} = (\mathcal{V}^{\sigma})^{T_p}.
$$
In a similar manner to what we did above for $T_{\ell}, T_r$, we can define lower and upper 
operators for words in $K_{\ell}, K_r$, that is, for words of the form $\tau(K_{\ell}, K_r)$. \\

\noindent
All operators of the form $\mathcal{V}\longrightarrow \mathcal{V}^{\tau}, \mathcal{V}_{\tau}, 
\mathcal{V}^{\tau(K_{\ell},K_r)}, \mathcal{V}_{\tau(K_{\ell}, K_r)}$ are {\em order-preserving}.\\

\noindent
By induction, from Lemma 2.1 (i), (ii) and the above definitions, we have the important fact
\begin{equation}
\mathcal{B}_{\tau} = \mathcal{B}\; \mbox{for all}\; \tau \in \Theta^1.
\end{equation}
The relation $\leq$ defined by
$$
\sigma \leq \tau \quad \mbox{if} \quad |\sigma|  > |\tau| \quad \mbox{or}
\quad \sigma = \tau \qquad\quad (\sigma,\tau \in \Theta^1)
$$
is easily seen to be a partial order which can be depicted as
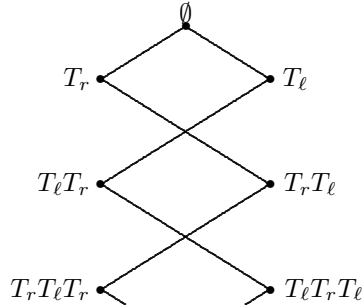
\[ \begin{minipage}[t]{12cm}
\beginpicture
\setcoordinatesystem units <0.8truecm,0.5truecm>
\setplotarea x from 0 to 12, y from 0 to 9
\setlinear
\plot 4.6 1  6 2.4 /
\plot 7.4 1  6 2.4 /
\plot 6 2.4  4.6 3.8 /
\plot 6 2.4  7.4 3.8 /
\plot 4.6 3.8  6 5.2 /
\plot 7.4 3.8  6 5.2 /
\plot 6 5.2  4.6 6.6 /
\plot 6 5.2  7.4 6.6 /
\plot 4.6 6.6  6 8 /
\plot 7.4 6.6  6 8 /

\plot 4.6 1  5 0.6 /
\plot 7.4 1  7 0.6 /

\put {\circle*{2.5}} at 4.65 1
\put {\circle*{2.5}} at 7.475 1
\put {\circle*{2.5}} at 4.65 3.8
\put {\circle*{2.5}} at 7.475 3.8
\put {\circle*{2.5}} at 4.65 6.6
\put {\circle*{2.5}} at 7.475 6.6
\put {\circle*{2.5}} at 6.075 8

\put {$T_rT_{\ell}T_r$} at 3.75 1
\put {$T_{\ell}T_r$} at 3.975 3.8
\put {$T_r$} at 4.2 6.6
\put {$\emptyset$} at 6 8.35
\put {$T_{\ell}$} at 7.825 6.6
\put {$T_rT_{\ell}$} at 8.075 3.8
\put {$T_{\ell}T_rT_{\ell}$} at 8.3 1
\endpicture
\end{minipage} \]


\begin{lemma}
\hspace*{-0.5cm}{\rm ([RZ2000] Lemma 4.11).} \ \,For any
$\mathcal{V} \in \mathcal{L}(\mathcal{CR})$ and $\sigma \in \Theta^1,$
we have
$$
\mathcal{V}^\sigma \cap \mathcal{V}^K = \mathcal{V}^{\sigma (K_{\ell},
K_r)}\,.
$$
\end{lemma}

\section{A Representation of $[\mathcal{S}, \mathcal{CR}]$}

In this section we recast an important representation of the interval 
$[\mathcal{S}, \mathcal{L}(\mathcal{CR})]$ due to Pol\'{a}k.  The 
objectives are twofold.  The original result is cast in the language of fully invariant congruences.  So the first goal is to rephrase the result in 
the language of varieties.  The second goal is to simplify a critical condition that appears in the hypothesis.\\

Although it is more intuitive to work with varieties, to do so sometimes requires considering 
more cases.  For instance,  $[\mathcal{LNB}, L\mathcal{RRO}]$  is a $T_r$-class and therefore, 
for any $\mathcal{V} \in [\mathcal{LNB}, L\mathcal{RRO}]$, we have $\mathcal{V}_{T_r} = 
\mathcal{LNB}$.  On the other hand, the corresponding operator on congruences leads to 
the fully invariant congruence corresponding to $\mathcal{LZ}$.  This divergence only happens 
low in the interval $[\mathcal{S, CR}]$ and does not happen above $\mathcal{LRB}$.  But it does 
require careful monitoring.  A similar 
situation prevails in regard to $T_{\ell}$-classes.  It is this feature that requires us to 
give special treatment to $\mathcal{T, LZ, RZ, S, LNB, RNB}$ in the notation below.

Let
$$
\mathcal{K}_0 = \{\mathcal{V}_K \mid \mathcal{V} \in
\mathcal{L}(\mathcal{CR})\}.
$$
The best known examples of varieties in $\mathcal{K}_0$ are all varieties in $\mathcal{L}(\mathcal{G})$ 
together with the varieties in $\mathcal{L}(\mathcal{CS})\backslash 
\mathcal{L}(\mathcal{R}e\mathcal{G})$ as well as varieties of the form $H\mathcal{U}, (\mathcal{U} 
\in \mathcal{L}(\mathcal{G}))$  and $D\mathcal{U}, (\mathcal{U} 
\in \mathcal{L}(\mathcal{CS})\backslash \mathcal{L}(\mathcal{R}e\mathcal{G}))$, see Pol\'{a}k [Po2], 
Theorem 2.6. In addition, Kadourek [K] has shown that 
$\mathcal{V}^T \in \mathcal{K}_0$ for all $\mathcal{V} \in [\mathcal{S}, \mathcal{CR}]$.  Little has been 
added to this list since then. \\

We know that $K$ is a complete congruence on the lattice $\mathcal{L}(\mathcal{CR})$ so that the quotient $\mathcal{L}(\mathcal{CR})/K$ 
is also a complete lattice.  Now $\mathcal{K}_0$ consists of all the elements in $\mathcal{L}(\mathcal{CR})$ that are least in their $K$-class.  Consequently, 
$\mathcal{K}_0$ and $\mathcal{L}(\mathcal{CR})/K$ are isomorphic as {\em ordered sets}.  Therefore, $\mathcal{K}_0$ is a complete lattice with 
respect to the induced order from $\mathcal{L}(\mathcal{CR})$.
Although $\mathcal{K}_0$ is not a $\cap$-sublattice of $\mathcal{L}(\mathcal{CR})$, it is 
useful to know that $\mathcal{K}_0$ is a complete $\vee$-sublattice of
$\mathcal{L}(\mathcal{CR}).$  Let
$$
\mathbb{N}_3^* = \{T^{\ast},L^{\ast},R^{\ast}\}\, \mbox{and} \; 
 \mathbb{N}_9 = \{\{T^{\ast}, L^{\ast}, 
R^{\ast}, \mathcal{T, LZ, RZ, S, LNB, RNB}\}.
$$
We set  
$$
\mathcal{K} = \mathcal{K}_0 \cup \{\mathbb{N}_{\,3}^*\}
$$
and endow $\mathcal{K}$ with the order inherited from $\mathcal{L}
(\mathcal{CR})$ in $\mathcal{K}_0$ together with
$$
T^{\ast} < L^{\ast} < \mathcal{V}\,, \quad T^{\ast} < R^{\ast} < \mathcal{V} \qquad \mbox{for
all $\mathcal{V} \in \mathcal{K}_0.$}
$$
\begin{lemma} {\rm (Pol\'{a}k [Po1], [Po2])}
$\mathcal{K}$ is a complete lattice.
\end{lemma}


\begin{notation}\label{not:XIII.3.4}
For $\mathcal{V} \in \mathcal{L}(\mathcal{CR}) \cup \mathbb{N}_{\,3}^*$
we define

\begin{minipage}[t]{7.75cm}
$$
\mathcal{V}_{K^*} = \left\{ \begin{array}{lll}
\mathcal{V}_K & \mbox{if} & \;\mathcal{V} \in \mathcal{L}(\mathcal{CR})
\backslash \mathbb{N}_9 \\[0.3cm]
L^{\ast} & \mbox{if} & \;\mathcal{V} \in \{L^{\ast}, \mathcal{LZ}, \mathcal{LNB}\}
\\[0.3cm]
T^{\ast} & \mbox{if} & \;\mathcal{V} \in \{T^{\ast}, \mathcal{T}, \mathcal{S}\}
\\[0.3cm]
R^{\ast} & \mbox{if} & \;\mathcal{V} \in \{R^{\ast}, \mathcal{RZ}, \mathcal{RNB}\}\,
\end{array} \right.
$$
\end{minipage} \ \
\begin{minipage}[t]{7.5cm}
$$
\;\mathcal{V}^{K^*} = \left\{ \begin{array}{lll}
\mathcal{V}^K  & \mbox{if} & \;\mathcal{V} \in
\mathcal{L}(\mathcal{CR}) \qquad\quad\;\, \\[0.3cm]
\mathcal{LNB} & \mbox{if} & \;\mathcal{V} = L^{\ast} \\[0.3cm]
\mathcal{S} & \mbox{if} & \;\mathcal{V} = T^{\ast} \\[0.3cm]
\mathcal{RNB} & \mbox{if} & \;\mathcal{V} = R^{\ast}\,.
\end{array} \right.
$$
\end{minipage}
\vskip0.5cm
\noindent
For any $\mathcal{V} \in \mathcal{L}(\mathcal{CR}),$ respectively
$\mathcal{V} \in \mathcal{L} (\mathcal{CR}) \cup \mathbb{N}_{\,3}^*,$
and $\tau \in \Theta^1,$ we write
$\mathcal{V}_{\tau K^*} = \big(\mathcal{V}_{\tau}\big)_{K^*},$ respectively
$\mathcal{V}^{\tau K^*} = \big(\mathcal{V}^{\tau}\big)^{K^*}.$
\end{notation}


Let $\mathbb{N} = \{1, 2, 3, \ldots\}$, the set of all positive integers and let $\Lambda = (\{0, 1\} \times \mathbb{N})\cup \{(0,0) = (1, 0)\}.$  Define a partial order on $\Lambda$ as follows:
$$
(i, m) < (j, n) \Longleftrightarrow m > n.
$$

\begin{notation}
Let $\Phi_1$ denote the set of all mappings $\varphi : \, \Lambda 
\longrightarrow \mathcal{K}$ satisfying the following conditions:
\begin{enumerate}
\item[{\rm (Q1)}] $(0, 0)\varphi \in \mathcal{K}_0\vspace*{0.05cm},$
\item[{\rm (Q2)}] $\varphi$ is order preserving\vspace*{0.05cm},
\item[{\rm (Q3)}] $(1, m) \varphi \neq L^{\ast},$\vspace*{0.05cm},
\item[{\rm (Q4)}] $(0, m) \varphi \neq R^{\ast},$ then\vspace*{0.05cm},
\item[{\rm (Q5)}] if $(i, m) \in \Lambda, (i, m)\varphi \in \mathcal{K}_0, k \geq 1, j\in \{0, 1\}, 
i + j \equiv k\;\mbox{mod}\; 2$ and $\tau\in \Theta$ is the unique element of length $k$ 
ending in $T_r$ if $j = 0$ and ending in $T_{\ell}$ if $j = 1$, then 
$((i, m)\varphi)_{\tau K^{\ast}} \leq (j, m + k)\varphi$.   
\end{enumerate}
\end{notation}

\begin{theorem}  
$\Phi_1$ is a complete sublattice of $\mathcal{K}^{\Lambda}.$
\end{theorem}
\begin{proof}
See  ([Po2], Lemma 3.4).
\end{proof}
\vskip0.1cm
The following definitions are the varietal equivalent of certain operators on fully invariant 
congruences introduced by Pol\'{a}k ([Po2], page 258].  Let $\mathcal{V} \in \mathcal{L}(\mathcal{CR})$.  
Then

\parbox{6cm}{
\[\mathcal{V}_0 = \left\{ \begin{array}{lll}
\mathcal{V}_{T_r}  & \mbox{if}\; \mathcal{LRB} \subseteq \mathcal{V}\\
\mathcal{LZ}         & \mbox{if}\; \mathcal{V} \in [\mathcal{LZ}, L\mathcal{RRO}]\\
\mathcal{T}           & \mbox{if}\; \mathcal{V} \in [\mathcal{T}, \mathcal{RRO}]
\end{array}
\right.  \]
} \  ;  \
\parbox{6cm}{
\[\mathcal{V}_1 = \left\{ \begin{array}{lll}
\mathcal{V}_{T_{\ell}}  & \mbox{if}\; \mathcal{RRB} \subseteq \mathcal{V}\\
\mathcal{RZ}         & \mbox{if}\; \mathcal{V} \in [\mathcal{RZ}, L\mathcal{LRO}]\\
\mathcal{T}           & \mbox{if}\; \mathcal{V} \in [\mathcal{T}, \mathcal{LRO}].
\end{array}
\right.  \]
}
\vspace{0.5cm}

Somewhat more simply, we define
$$
\mathcal{V}^0 = \mathcal{V}^{T_r};\;\;\mathcal{V}^1 = \mathcal{V}^{T_{\ell}}.
$$
\noindent
The lower operators $\mathcal{V} \longrightarrow \mathcal{V}_{(i, m)} ((i, m) \in \Lambda)$ are then defined 
inductively by 
$$
\mathcal{V}_{(0,0)} = \mathcal{V}_{(1,0)} = \mathcal{V};\;\;\mathcal{V}_{(i,m)} = 
(\mathcal{V}_{(1-i, m-1)})_i\;(\mbox{ if} \; m \geq 1),
$$
while the upper operators $\mathcal{V} \longrightarrow \mathcal{V}^{(i, m)} ((i, m) \in \Lambda)$ are 
defined inductively by 
$$
\mathcal{V}^{(0,0)} = \mathcal{V}^{(1,0)} = \mathcal{V};\;\;\mathcal{V}^{(i,m)} = 
(\mathcal{V}^i)^{(1-i, m-1)}) \;(\mbox{ if} \; m \geq 1).
$$
\begin{theorem}{\rm ([Po2], Theorem 3.6)}  For
every $\mathcal{V} \in [\mathcal{S}, \mathcal{CR}],$ define a function
$\pi_{_{\scriptstyle \mathcal{V}}}$ as follows:  for all $(i,m) \in \Lambda$


\[
\pi_{_{\scriptstyle\mathcal{V}}}: \, (i,m) \longrightarrow \left\{ \begin{array}{ll} 
\mathcal{V}_K &\;\mbox{if $m = 0$}\\
\mathcal{V}_{(i,m) K^{\ast}}& \; 
\mbox{otherwise}.
\end{array}
\right.\]
The mappings
$$
\pi\!: \,\mathcal{V} \longrightarrow \pi_{_{\scriptstyle\mathcal{V}}}\,,
\quad\; \psi\!: \,\varphi \longrightarrow \bigcap\limits_{(i,m) \in \Lambda}
((i,m) \varphi )^{K^*(i,m)}
$$
are mutually inverse ismorphisms between the lattices $[\mathcal{S},
\mathcal{CR}]$ and $\Phi_1 .$
\end{theorem}
\vspace{0.5cm}

\section{An Alternative Representation}

Up to this point, we have simply presented the results of Pol\'{a}k in 
the language of varieties with very minor modifications.  See also [Pa] for a similar 
treatment.   We now wish to replace $\Lambda$ 
by $\Theta^1$ and simplify the condition (Q5).\\

It is a simple exercise to see that $\Lambda$ and $\Theta^1$ are actually 
isomorphic under the mapping $\gamma$ defined  on $\Theta^1$ as follows: for $\tau \in \Theta^1$
$$
\gamma:\;\tau \longrightarrow \left\{ \begin{array}{ll}
(0,0) = (1,0) & \quad \mbox{if}\;\; \tau = \emptyset\,,
\\[0.2cm]
(0, m) & \quad \mbox{if}\;\; t(\tau) = T_r, |\tau| = m \geq 1\,,
\\[0.2cm]
(1, m) & \quad \mbox{if}\;\; t(\tau) = T_{\ell}, |\tau| = m \geq 1.
\end{array} \right.
$$

 In plain words, $\tau\gamma = (i, n)$ where $i$ determines $t(\tau)$ and $n$ determines 
the length of $\tau$.  
Since $T_{\ell}$ and $T_r$ alternate in the word $\tau$, the pair $(i,m)$ uniquely determines $\tau$, 
from back to front, so to speak.

\begin{lemma}
Let $\mathcal{V} \in [\mathcal{S, CR}], \tau\in \Theta^1$ and $\tau\gamma = (i, m)$. \\
{\rm (i)} $\mathcal{V}^{\overline{\tau}} = \mathcal{V}^{(i, m)}$.\\
{\rm (ii)} $\mathcal{V}_{\tau} \notin \{\mathcal{S, LNB, RNB}\} \Longrightarrow \mathcal{V}_{\tau} = 
\mathcal{V}_{(i, m)}$.
\end{lemma}
\begin{proof}
This follows by straightforward induction from the definitions of $\mathcal{V}_0, 
\mathcal{V}_1$, $\mathcal{V}^0, \mathcal{V}^1$.
\end{proof}

\begin{notation} 
Let $\Phi$ denote the set of all mappings $\varphi : \,\Theta^1
\longrightarrow \mathcal{K}$ satisfying the following conditions:
\begin{enumerate}
\item[{\rm (P1)}] $\emptyset \varphi \in \mathcal{K}_0\vspace*{0.05cm},$
\item[{\rm (P2)}] $\varphi$ is order preserving\vspace*{0.05cm},
\item[{\rm (P3)}] if \,$\tau \in \Theta$ and $\tau \varphi = L^{\ast},$ then
$t(\tau ) = T_r\vspace*{0.05cm},$
\item[{\rm (P4)}] if \,$\tau \in \Theta$ and $\tau \varphi = R^{\ast},$ then
$t(\tau ) = T_{\ell}\vspace*{0.05cm},$
\item[{\rm (P5)}] if \,$\tau \in \Theta,$ then $(\emptyset
\varphi)_{\tau K^*} \leq \tau \varphi \vspace*{0.05cm},$
\item[{\rm (P6)}] if \,$\sigma ,\tau \in \Theta$ are such that $\sigma
\varphi \in \mathcal{K}_0$ and $t(\sigma ) \neq h(\tau ),$ then
$(\sigma \varphi )_{\tau K^*} \subseteq (\sigma \tau ) \varphi .$
\end{enumerate}
\end{notation}
\vskip0.2cm

\begin{proposition}
The mapping $\eta:\; \varphi \longrightarrow \gamma\varphi\; (\varphi \in \Phi_1)$ is an order isomorphism 
of $\Phi_1$ onto $\Phi$ with inverse $\eta^{-1}: \theta \longrightarrow \gamma^{-1}\theta\;(\theta 
\in \Phi)$.
\end{proposition}
One way of interpreting this proposition intuitively would be to say that if we identify $\Lambda$ and $\Theta^1$ via the isomorphism $\eta$, 
then $\Phi_1 = \Phi$.

\begin{proof}
Let $\varphi \in \Phi_1$.  The first goal is to show that $\gamma\varphi \in \Phi$.  To begin 
with we have $\emptyset\gamma\varphi = (0,0)\varphi \in \mathcal{K}_0$ so 
that $\gamma\varphi$ satisfies (P1).  Let $\sigma, \tau \in \Theta^1, \sigma \leq \tau$.  Since 
$\gamma$ and $\varphi$ are order preserving, $\sigma\gamma \leq \tau\gamma$ and $\sigma\gamma\varphi \leq 
\tau\gamma\varphi$.  Thus $\gamma\varphi$ is order preserving, that is, satisfies (P2).  Now suppose that 
$\tau \in \Theta$ and $\tau\gamma\varphi = L^{\ast}$.  By (Q3), this implies that $\tau\gamma = (0, |\tau|)$ and 
therefore, by the definition of $\gamma$, we have $t(\tau) = T_r$.  Hence, $\gamma\varphi$ satisfies (P3).  That 
$\gamma\varphi$ satisfies (P4) follows by duality. \\

Now consider (P5).  Let $\tau \in \Theta$ and $|\tau| = k \geq 1$.  Without loss of generality, 
we may assume 
that $t(\tau) = T_r$ so that $\tau\gamma = (0,k)$.  First suppose that $k$ is even.  
Then, with $i = j = m = 0$, we can invoke (Q5), 
to obtain
$$
(\emptyset\gamma\varphi)_{\tau K^{\ast}}  = ((0, 0)\varphi)_{\tau K^{\ast}} \leq (0, k)\varphi 
 = \tau\gamma\varphi.
$$
On the other hand, if $k$ is odd, then we may take $i = 1, m = 0, j = 0$ and again invoke (Q5) to 
obtain
$$
(\emptyset\gamma\varphi)_{\tau K^{\ast}}  = ((1, 0)\varphi)_{\tau K^{\ast}} \leq (0, k)\varphi 
 = \tau\gamma\varphi.
$$
Thus $\varphi$ satisfies (P5).  \\

Finally, (P6).  Let $\sigma ,\tau \in \Theta$ be such that $\sigma\gamma\varphi 
\in \mathcal{K}_0$ and $t(\sigma ) \neq h(\tau )$.  This implies that $|\sigma\tau| = |\sigma| 
+ |\tau|$.  
Let $|\sigma| = m, |\tau| = k$.  Then 
$m, k \geq 1$.  Let $\sigma\gamma = (i, m), \tau\gamma = (j, k)$.  We break the argument into 
four cases determined by the parity of $i, j$, that is, the values of $i, j \in \{0, 1\}$.\\

Case: $i = 0, j = 0$.  Then $t(\sigma) = t(\tau) = t(\sigma\tau) = T_r$.  Hence 
$\sigma\tau\gamma = (0, m + k)$.  Since
$T_r = t(\sigma) \neq h(\tau)$, it follows 
that $h(\tau) = T_{\ell}$. Consequently, $k$ must be even and $i + j \equiv k$.  We can now 
invoke $(Q5)$ to conclude that
$$
(\sigma\gamma\varphi)_{\tau K^{\ast}} = ((i, m)\varphi)_{\tau K^{\ast}} \leq (j, m + k)\varphi = 
(0, m + k)\varphi = (\sigma\tau)\gamma\varphi
$$
and (P6) is satisfied.\\

Case: $i = 0, j = 1$.  Then $t(\sigma) = T_r,  t(\tau) = t(\sigma\tau) = T_{\ell}$.  Hence 
$\sigma\tau\gamma = (1, m + k)$.  Since $T_r = t(\sigma) \neq h(\tau)$, it follows that 
$h(\tau) = T_{\ell}$ and therefore that $k$ is odd.  Consequently, $i + j \equiv k$ and 
we may invoke (Q5):
$$
(\sigma\gamma\varphi)_{\tau K^{\ast}} = ((i, m)\varphi)_{\tau K^{\ast}} \leq (j, m + k)\varphi = 
(\sigma\tau)\gamma\varphi
$$
Case: $i = 1, j = 0$.  Then $t(\sigma) = T_{\ell},  t(\tau) = t(\sigma\tau) = T_r$.  Hence 
$\sigma\tau\gamma = (0, m + k)$.  Since $T_{\ell} = t(\sigma) \neq h(\tau)$, it follows that 
$h(\tau) = T_r$ and therefore that $k$ is odd.  Consequently, $i + j \equiv k$ and 
we may invoke $(Q5)$:
$$
(\sigma\gamma\varphi)_{\tau K^{\ast}} = ((i, m)\varphi)_{\tau K^{\ast}} \leq (j, m + k)\varphi = 
(\sigma\tau)\gamma\varphi
$$
Case: $i = 1, j = 1$.  Then $t(\sigma) = t(\tau) = t(\sigma\tau) = T_{\ell}$. Hence 
$\sigma\tau\gamma = (1, m + k)$.  Since $T_{\ell} = t(\sigma) \neq h(\tau)$, it follows that 
$h(\tau) = T_r$ and therefore that $k$ is even.  Consequently, $i + j \equiv k$ and 
we may invoke (Q5):
$$
(\sigma\gamma\varphi)_{\tau K^{\ast}} = ((i, m)\varphi)_{\tau K^{\ast}} \leq (j, m + k)\varphi = 
(\sigma\tau)\gamma\varphi
$$
We have now shown that $\gamma\varphi$ satisfies the conditions (P1) - (P6) and therefore 
that $\gamma\varphi \in \Phi$.\\

Now for the converse.  Let $\theta\in \Phi$ and consider $\gamma^{-1}\theta$.    
Then $(0,0)\gamma^{-1}\theta = \emptyset\theta\in \mathcal{K}_0$ so that $\gamma^{-1}\theta$ 
satisfies (Q1).   Since $\gamma$ is an order isomorphism, so also is $\gamma^{-1}$ from which it 
immediately follows that $\gamma^{-1}\theta$ is order preserving and satisfies (Q2).  Now 
let $(1,m)\in \Lambda$.  If $m= 0$, then $(1, 0)\gamma^{-1}\theta = \emptyset\theta \in \mathcal{K}_0$ 
and therefore does not equal $L^{\ast}$.  If $m \geq 1$, then $t((1, m)\gamma^{-1}) = T_{\ell}$ and, 
by (P3), we have 
$(1, m)\gamma^{-1}\theta \neq L^{\ast}$.
Thus $\gamma^{-1}\theta$ satisfies (Q3) and, by duality, 
(Q4).\\

Now consider the condition (Q5).  Let $(i,m) \in \Lambda, 
(i, m)\gamma^{-1}\theta \in \mathcal{K}_0, k \geq 1, j\in \{0, 1\}$ and $i + j \equiv k$.  
Let $\sigma = (i, m)\gamma^{-1}$.  Then $|\sigma| = m$.   \\

Case: $m = 0$.  Then $\sigma = \emptyset$.  Let $\tau = (j, m + k)\gamma^{-1} = (j, k)\gamma^{-1}$.  Since 
$k \geq 1$, we have $\tau \in \Theta$.  Then, by (P5), 
$$
((i, m)\gamma^{-1}\theta)_{\tau K^{\ast}}= (\sigma\theta)_{\tau K^{\ast}} = (\emptyset\theta)_{\tau K^{\ast}} \leq \tau\theta 
=(j, m + k)\gamma^{-1}\theta
$$
as required.  \\

Case: $m \geq 1$.  Let $\tau = (j, k)\gamma^{-1}$.  \\

Subcase: $i = 0 = j $.  Then $t(\sigma) = T_r$.  Hence,
$$
i =j = 0 \Longrightarrow k \equiv 0, t(\tau) = T_r \Longrightarrow h(\tau) = T_{\ell} \neq 
t(\sigma) \Longrightarrow \sigma\tau\gamma = (0, m + k) = (j, m + k).
$$
 By (P6), this implies that $(\sigma\theta)_{\tau K^{\ast}} \leq (\sigma\tau)\theta$.  Hence
$$
((i, m)\gamma^{-1}\theta)_{\tau K^{\ast}} = (\sigma\theta)_{\tau K^{\ast}} \leq (\sigma\tau)\theta 
= (j, m + k)\gamma^{-1}\theta
$$
as required.  \\

Subcase: $i = 0, j = 1$.   Then $k$ is odd so that 
\begin{eqnarray*}
i = 0, j = 1 &\Longrightarrow& t(\tau) = T_{\ell} = h(\tau) \neq T_r = t(\sigma) \Longrightarrow (\sigma\tau)\gamma 
= (1, m + k) = (j, m + k)\\ &\Longrightarrow& t(\sigma\tau)= t(\tau) = T_{\ell}.
\end{eqnarray*}
Therefore, 
$$
((i, m)\gamma^{-1}\theta)_{\tau K^{\ast}} = (\sigma\theta)_{\tau K^{\ast}} \leq (\sigma\tau)\theta =
(j, m + k)\gamma^{-1}\theta
$$
as required.\\

Subcase:$i = 1,  j = 0$.  Then
\begin{eqnarray*}
i = 1, j = 0 &\Longrightarrow& k \equiv 1, t(\sigma) = T_{\ell}, t(\tau) = T_r\\
&\Longrightarrow& h(\tau) = T_r \neq t(\sigma)\\
&\Longrightarrow& |\sigma\tau| = |\sigma| + |\tau| = m + k
\end{eqnarray*}
Also $t(\sigma\tau) = t(\tau) = T_r.$
Hence $(\sigma\tau)\gamma = (0, m + k) = (j, m + k)$.  Therefore once again
$$
((i, m)\gamma^{-1}\theta)_{\tau K^{\ast}} = (\sigma\theta)_{\tau K^{\ast}} \leq (\sigma\tau)\theta =
(j, m + k)\gamma^{-1}\theta
$$
as required.  \\

Case:$i = 1, j = 1$.  We have
\begin{eqnarray*}
i = 1 = j &\Longrightarrow& t(\sigma) = t(\tau) = T_{\ell}\\
&\Longrightarrow& k \equiv 0, h(\tau) = T_r \neq t(\sigma)\\
&\Longrightarrow& |\sigma\tau| = |\sigma| + |\tau| = m + k\\
&\Longrightarrow& t(\sigma\tau) = t(\tau) = T_{\ell}\\
&\Longrightarrow& (\sigma\tau)\gamma = (1, m + k). 
\end{eqnarray*}
Hence we once again have 
$$
((i, m)\gamma^{-1}\theta)_{\tau K^{\ast}} = (\sigma\theta)_{\tau K^{\ast}} \leq (\sigma\tau)\theta =
(j, m + k)\gamma^{-1}\theta
$$
as required and $\gamma^{-1}\theta$ satisfies (Q5).  Therefore $\gamma^{-1}\theta \in \Phi_1$.\\

Now that we know that $\eta$ maps $\Phi_1$ to $\Phi$ and that the mapping $\eta^{\ast}:
\theta \longrightarrow \gamma^{-1}\theta$ maps $\Phi$ to $\Phi_1$ it is easy to check that $\eta$ and 
$\eta^{\ast}$ are inverse isomorphisms.
\end{proof}

\begin{theorem}  
$\Phi$ is a complete sublattice of $\mathcal{K}^{\Theta^1}.$
\end{theorem}
\begin{proof}
This follows from Theorem 3.4 and Proposition 4.3.
\end{proof}
\vskip0.1cm

The next theorem establishes an isomorphism between the lattices $[\mathcal{S},
\mathcal{CR}]$ and $\Phi.$  This is in terms of $\Theta^1$ and $\Phi$ rather 
than $\Lambda$ and $\Phi_1$.  This appeared without proof in [R1990].

\begin{theorem} \ For
every $\mathcal{V} \in [\mathcal{S}, \mathcal{CR}],$ define a function
$\chi_{_{\scriptstyle \mathcal{V}}}$ by
$$
\chi_{_{\scriptstyle\mathcal{V}}}: \, \tau \longrightarrow \left\{
\begin{array}{lll}
\mathcal{V}_K & \quad\mbox{if} & \;\tau = \emptyset
\\[0.2cm]
\mathcal{V}_{\tau K^*} & \quad\mbox{if} & \;\tau \in \Theta.
\end{array} \right. 
$$
The mappings
$$
\chi\!: \,\mathcal{V} \longrightarrow \chi_{_{\scriptstyle\mathcal{V}}}\,,
\quad\; \xi\!: \,\varphi \longrightarrow \bigcap\limits_{\tau \in \Theta^1}
(\tau \varphi )^{K^*\overline{\tau}}
$$
are mutually inverse ismorphisms between the lattices $[\mathcal{S},
\mathcal{CR}]$ and $\Phi .$
\end{theorem}
\begin{proof}
Since, by Theorem 3.5,   $\pi: [\mathcal{S}, \mathcal{CR}]\longrightarrow \Phi_1$ 
is an isomorphism and, by Proposition 4.3,  
$\eta:\Phi_1 \longrightarrow \Phi$ is also an  
isomorphism, it follows that $\pi\eta:[\mathcal{S}, \mathcal{CR}] \longrightarrow \Phi$ is 
also an isomorphism with inverse $(\pi\eta)^{-1} = \eta^{-1}\pi^{-1}$.  So it will suffice to show 
that $\chi = \pi\eta$ and that $\xi = \eta^{-1}\pi^{-1}$.\\

Let $\mathcal{V} \in [\mathcal{S, CR}]$ and $\tau \in \Theta^1$.  Then, $(\mathcal{V})\pi\eta = 
\pi_{\mathcal{V}}\eta = \gamma\pi_{\mathcal{V}}$ so that, with $\tau = \emptyset$, 
$$
\emptyset(\mathcal{V})\pi\eta = \emptyset\gamma\pi_{\mathcal{V}} = (0,0)\pi_{\mathcal{V}} = 
\mathcal{V}_K = \emptyset\chi_{_{\mathcal{V}}}.
$$
Now let $|\tau| = 1$, say $\tau = T_r$.  From the definition of $\mathcal{V}_0$, and the assumption that 
$\mathcal{V} \in [\mathcal{S}, \mathcal{CR}]$, we have
\[\mathcal{V}_{0K^{\ast}} = \left\{ \begin{array}{lll}
\mathcal{V}_{T_rK}  & \mbox{if}\; \mathcal{LRB} \subseteq \mathcal{V},\\
\mathcal{LZ}_{K^{\ast}}         & \mbox{if}\; \mathcal{V} \in [\mathcal{LNB}, L\mathcal{RRO}],\\
\mathcal{T}_{K^{\ast}}           & \mbox{if}\; \mathcal{V} \in [\mathcal{S}, \mathcal{RRO}],
\end{array}
\right.  \]

\[ = \left\{ \begin{array}{lll}
\mathcal{V}_{T_rK}  & \mbox{if}\; \mathcal{LRB} \subseteq \mathcal{V},\\
L^{\ast}         & \mbox{if}\; \mathcal{V} \in [\mathcal{LNB}, L\mathcal{RRO}],\\
T^{\ast}           & \mbox{if}\; \mathcal{V} \in [\mathcal{S}, \mathcal{RRO}].
\end{array}
\right.  \]
On the other hand

\[\mathcal{V}_{T_rK^{\ast}} = \left\{ \begin{array}{lll}
\mathcal{V}_{T_rK^{\ast}}  & \mbox{if}\; \mathcal{LRB} \subseteq \mathcal{V},\\
\mathcal{LNB}_{K^{\ast}}         & \mbox{if}\; \mathcal{V} \in [\mathcal{LNB}, L\mathcal{RRO}],\\
\mathcal{S}_{K^{\ast}}           & \mbox{if}\; \mathcal{V} \in [\mathcal{S}, \mathcal{RRO}],
\end{array}
\right.  \]

\[ = \left\{ \begin{array}{lll}
\mathcal{V}_{T_rK}  & \mbox{if}\; \mathcal{LRB} \subseteq \mathcal{V},\\
L^{\ast}         & \mbox{if}\; \mathcal{V} \in [\mathcal{LNB}, L\mathcal{RRO}],\\
T^{\ast}           & \mbox{if}\; \mathcal{V} \in [\mathcal{S}, \mathcal{RRO}].
\end{array}
\right.  \]
Thus $\mathcal{V}_{0K^{\ast}} = \mathcal{V}_{T_rK^{\ast}}$ and dually, we have 
$\mathcal{V}_{1K^{\ast}} = \mathcal{V}_{T_{\ell}K^{\ast}}$.   Now consider $\tau \in 
\Theta, |\tau| \geq 2$.  If $\mathcal{V}_{\tau} \notin \{\mathcal{S, LNB, RNB}\}$, then by 
Lemma 4.1 we have $\mathcal{V}_{\tau} = \mathcal{V}_{(i, m)}$ and therefore 
$\mathcal{V}_{\tau K^{\ast}} = \mathcal{V}_{(i, m)K^{\ast}}$.  So, let us now assume that 
$\mathcal{V}_{\tau} \in \{\mathcal{S, LNB, RNB}\}$.  Let $\sigma \in \Theta^1$ be the longest initial 
segment of $\tau$ such that $\mathcal{V}_{\sigma} \notin \{\mathcal{S, LNB, RNB}\}$ and 
$\rho \in \Theta$ be such that $\tau = \sigma\rho$.  First assume that $\rho = T_r$.  Since 
$\mathcal{V}_{\sigma T_r} \in \{\mathcal{S, LNB, RNB}\}$ and $\mathcal{V}_{\sigma T_r}$ is 
the least element in its $T_r$-class, we must have $\mathcal{V}_{\sigma T_r} \in 
\{\mathcal{S, LNB}\}$. \\

Case: $\mathcal{V} _{\sigma T_r} = \mathcal{LNB}$.   By Lemma 4.1,  $\mathcal{V}_{\sigma} = 
\mathcal{V}_{(\sigma)\gamma}$.  Since $\mathcal{LNB}$ is the least element 
in its $T_r$-class $[\mathcal{LNB}, L\mathcal{RRO}]$, we must have $\mathcal{V}_{\sigma} = 
\mathcal{V}_{(\sigma)\gamma} \in [\mathcal{LNB}, L\mathcal{RRO}]$.  Hence, 
$(\mathcal{V}_{(\sigma)\gamma})_{(0,1)} =  (\mathcal{V}_{(\sigma)\gamma})_0 = \mathcal{LZ}$ 
so that 
$$
\mathcal{V}_{\sigma T_r K^{\ast}} = \mathcal{LNB}_{K^{\ast}} = L^{\ast} = \mathcal{LZ}_{K^{\ast}} 
= (\mathcal{V}_{(\sigma)\gamma})_{(0,1) K^{\ast}} = 
(\mathcal{V}_{(\sigma T_r)\gamma})_{K^{\ast}}. 
$$
Going one step further, we obtain
$$
\mathcal{V}_{\sigma T_rT_{\ell}} = \mathcal{S}\;\mbox{while}\; 
\mathcal{V}_{(\sigma T_rT_{\ell})\gamma} = (\mathcal{V}_{(\sigma)\gamma})_{(0,1)(1,1)} 
= \mathcal{LZ}_{(1,1)} = \mathcal{LZ}_1 =  \mathcal{T}.
$$
Consequently,
$$
\mathcal{V}_{\sigma T_rT_{\ell}K^{\ast}} = \mathcal{S}_{K^{\ast}} = T^{\ast} = \mathcal{T}_{K^{\ast}} 
= \mathcal{V}_{(\sigma T_rT_{\ell})\gamma K^{\ast}}.
$$
Continuing in this way, we see that for any $\rho \in \Theta$ with $|\rho| \geq 2$ we have 
$$
\mathcal{V}_{\sigma\rho K^{\ast}} = T^{\ast} 
= \mathcal{V}_{(\sigma\rho)\gamma K^{\ast}}.
$$
Thus for all $\tau \in \Theta^1$ we have that $\mathcal{V}_{\tau K^{\ast}}  
= \mathcal{V}_{(\tau)\gamma K^{\ast}}$.  Equivalently, we have now established that 
for $\tau\gamma = (i, m)$ we will have $\mathcal{V}_{\tau K^{\ast}} 
= \mathcal{V}_{(i, m)K^{\ast}}$.   Thus, for all $\tau \in \Theta$, we have
$$
\tau((\mathcal{V})\pi\eta) = ((\tau)\gamma)(\mathcal{V})\pi = (i, m)\pi_{\mathcal{V}} 
= \mathcal{V}_{(i, m)K^{\ast}} = \mathcal{V}_{\tau K^{\ast}} = \tau\chi_{_{\mathcal{V}}}
$$
while we have already seen that $(\emptyset)(\mathcal{V})\pi\eta = \emptyset\chi_{_{\mathcal{V}}}.$  
Therefore we have $(\mathcal{V})\pi\eta = \chi_{_{\mathcal{V}}} = (\mathcal{V})\chi$.  Consequently, 
$\pi\eta = \chi$ as required.\\

Now consider $\eta^{-1}\pi^{-1}$.  By Theorem 3.5, $\eta^{-1}\pi^{-1} = \eta^{-1}\psi$ 
where $\psi$ is as defined in Theorem 3.5.  
Let $\varphi \in \Phi$.  Then
\begin{eqnarray*}
\varphi(\eta^{-1}\pi^{-1}) &=& \varphi(\eta^{-1}\psi)\\
&=&  (\varphi\eta^{-1})\psi\\
&=& \bigcap\limits_{(i,m) \in \Lambda}((i,m) \varphi\eta^{-1} )^{K^*(i,m)}\\
&=& \bigcap\limits_{(i,m) \in \Lambda}(((i,m)\gamma^{-1}) \varphi )^{K^*(i,m)}\\
&=& \bigcap\limits_{(i,m) \in \Lambda}((\tau) \varphi )^{K^*(i,m)}\;\;\mbox{where $\tau\gamma = (i, m)$}\\
&=& \bigcap\limits_{\tau\in \Theta^1}((\tau) \varphi )^{K^*(i,m)}\;\;\mbox{since $\gamma$ is an 
isomorphism}\\
&=& \bigcap\limits_{\tau\in \Theta^1}((\tau) \varphi )^{K^*\overline{\tau}}\;\mbox{by Lemma 4.1}\\
&=&  \varphi\xi.
\end{eqnarray*}
Thus $\xi = \eta^{-1}\pi^{-1} = \chi^{-1}$, as required. 
\end{proof}

\noindent
\begin{notation}  {\rm (i)}  Throughout the remainder of this article, the symbols $\chi, \xi$ will have the 
meanings assigned to them in Theorem 4.5.\\
{\rm (ii)}  For any subset $\mathcal{P} \subseteq [\mathcal{S, CR}]$ we will write 
$$
\Phi_{\mathcal{P}} = \mathcal{P}\chi = \{\chi_{_{\mathcal{V}}}|\mathcal{V} \in \mathcal{P}\}
$$
\end{notation}

For $\varphi \in \Phi$, the definition of $\varphi \xi $ is, perhaps, a bit more meaningful  when like 
elements are grouped together, as in:
\begin{eqnarray}
\varphi \xi &=& (\emptyset \varphi )^K \cap \bigcap\limits_{\tau \varphi
\in \mathcal{K}_0} (\tau \varphi )^{K\overline{\tau}} \cap  
\bigcap\limits_{\tau \varphi = T^*}
\mathcal{S}^{\overline{\tau}}\cap
\bigcap\limits_{\tau \varphi = L^*} \mathcal{LNB}^{\overline{\tau}}
\cap \bigcap\limits_{\tau \varphi = R^*} \mathcal{RNB}^{\overline{\tau}}.
\end{eqnarray}
where terms are grouped together naturally according to the value of $\tau\varphi$.
We will refer to the expressions
$$
(\emptyset \varphi )^K, \bigcap\limits_{\tau \varphi
\in \mathcal{K}_0} (\tau \varphi )^{K\overline{\tau}}, \bigcap\limits_{\tau \varphi = T^*}
\mathcal{S}^{\overline{\tau}}, \bigcap\limits_{\tau \varphi = L^*} \mathcal{LNB}^{\overline{\tau}}, 
\bigcap\limits_{\tau \varphi = R^*} \mathcal{RNB}^{\overline{\tau}},
$$
as the {\em components} of $\varphi\xi$ and we will refer to the expression on the right hand 
side of (3) as the {\em component form} of $\varphi\xi$.   Of course, if $\varphi = \chi_{_{\mathcal{V}}}$, 
then $\mathcal{V} =  \chi_{_{\mathcal{V}}}\xi$ so that (3) gives us a description of the 
variety $\mathcal{V}$ in {\em component} form.  This will be very useful below.\\

\noindent  
Adopting the terminology of Pol\'{a}k, we will refer to the element 
$\chi_{\mathcal{V}}$ as the {\em ladder} of 
$\mathcal{V}$ and to any element of $\Phi$ as a {\em ladder}.  Of course, by Theorem 4.5, 
every ladder is the ladder of some variety in $[\mathcal{S, CR}]$.     Taking advantage of 
the representation of $\Theta$ as a partially ordered set and the fact that $\chi_{\mathcal{V}}$ is 
order-preserving, we can likewise display the values of $\chi_{\mathcal{V}}$ as the labels at 
the vertices of a ladder or network.

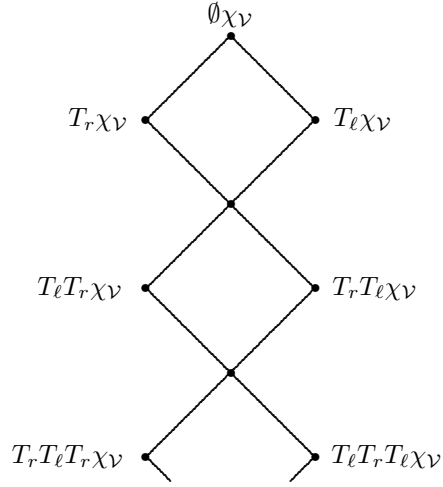
\[ \begin{minipage}[t]{12cm}
\beginpicture
\setcoordinatesystem units <0.8truecm,0.8truecm>
\setplotarea x from 0 to 12, y from 0 to 9
\setlinear
\plot 4.6 1  6 2.4 /
\plot 7.4 1  6 2.4 /
\plot 6 2.4  4.6 3.8 /
\plot 6 2.4  7.4 3.8 /
\plot 4.6 3.8  6 5.2 /
\plot 7.4 3.8  6 5.2 /
\plot 6 5.2  4.6 6.6 /
\plot 6 5.2  7.4 6.6 /
\plot 4.6 6.6  6 8 /
\plot 7.4 6.6  6 8 /

\plot 4.6 1  5 0.6 /
\plot 7.4 1  7 0.6 /

\put {\circle*{2.5}} at 4.65 1
\put {\circle*{2.5}} at 7.475 1
\put {\circle*{2.5}} at 6.075 2.4
\put {\circle*{2.5}} at 4.65 3.8
\put {\circle*{2.5}} at 7.475 3.8
\put {\circle*{2.5}} at 6.075 5.2
\put {\circle*{2.5}} at 4.65 6.6
\put {\circle*{2.5}} at 7.475 6.6
\put {\circle*{2.5}} at 6.075 8

\put {$T_rT_{\ell}T_r\chi_{\mathcal{V}}$} at 3.3 1
\put {$T_{\ell}T_r\chi_{\mathcal{V}}$} at 3.5 3.8
\put {$T_r\chi_{\mathcal{V}}$} at 3.8 6.6
\put {$\emptyset\chi_{\mathcal{V}}$} at 6 8.35
\put {$T_{\ell}\chi_{\mathcal{V}}$} at 8.2 6.6
\put {$T_rT_{\ell}\chi_{\mathcal{V}}$} at 8.4 3.8
\put {$T_{\ell}T_rT_{\ell}\chi_{\mathcal{V}}$} at 8.6 1
\endpicture
\end{minipage} \]

\begin{lemma}
Let $\mathcal{U, V} \in [\mathcal{S, CR}]$.\\
{\rm (i)} $\mathcal{U}\; K\; \mathcal{V} \Longleftrightarrow \emptyset\chi_{\mathcal{U}} = 
\emptyset\chi_{\mathcal{V}}$.\\
{\rm (ii)} $\mathcal{U}\; T_{\ell}\; \mathcal{V} \Longleftrightarrow \tau\chi_{\mathcal{U}} = 
\tau\chi_{\mathcal{V}}$ for all $\tau \in \Theta$ with $h(\tau) = T_{\ell}$.\\
{\rm (iii)} $\mathcal{U}\; K_{\ell}\; \mathcal{V} \Longleftrightarrow 
\emptyset\chi_{\mathcal{U}} = \emptyset\chi_{\mathcal{V}}$ and 
$\tau\chi_{\mathcal{U}} = 
\tau\chi_{\mathcal{V}}$ for all $\tau \in \Theta$ with $h(\tau) = T_{\ell}$.
\end{lemma}
\begin{proof}  (i)  This follows immediately from the fact that $\emptyset\chi_{_{\mathcal{V}}} = \mathcal{V}_K$ for all $\mathcal{V} \in \mathcal{L}(\mathcal{CR})$.  Part (iii) follows 
immediately from parts (i) and (ii) combined with the fact that $K_{\ell} = K\cap T_{\ell}$.  
So it remains to prove part (ii).   Let $\mathcal{U} \;T_{\ell}\; 
\mathcal{V}$ and $\tau \in \Theta$ be such that $h(\tau) = T_{\ell}$.  Then 
$\mathcal{U}_{T_{\ell}} = \mathcal{V}_{T_{\ell}}$ so that necessarily 
$\mathcal{U}_{\tau} = \mathcal{V}_{\tau}$ and therefore 
$\tau\chi_{_{\mathcal{U}}} = \tau\chi_{_{\mathcal{V}}}$.  That establishes the direct implication.  

Now assume that $\tau \chi_{_{\scriptstyle\mathcal{U}}} = \tau
\chi_{_{\scriptstyle\mathcal{V}}}$
for all $\tau \in \Theta $ with $h(\tau )=T_{\ell}.$ Then for all $\sigma
\noindent\in \Theta ,$ we get
$$
\sigma \chi_{_{\scriptstyle \mathcal{U}_{T_{\ell}}}} = \mathcal{U}_{T_{\ell}
\sigma K^*} = (T_{\ell}\sigma ) \chi_{_{\scriptstyle\mathcal{U}}} = (T_{\ell}
\sigma ) \chi_{_{\scriptstyle\mathcal{V}}} =
\mathcal{V}_{T_{\ell} \sigma K^*} = \sigma
\chi_{_{\scriptstyle\mathcal{V}_{T_{\ell}}}}\vspace*{0.1cm}\,.
$$

\noindent
In addition
$$
\mathcal{U}_{T_{\ell} K^*} = T_{\ell} \chi_{_{\scriptstyle\mathcal{U}}} = T_{\ell}
\chi_{_{\scriptstyle\mathcal{V}}}
= \mathcal{V}_{T_{\ell} K^*}\,.
$$
Hence
\begin{align*}
\mathcal{U}_{T_{\ell}} \in \mathcal{L}(\mathcal{B}) &\;\Longleftrightarrow\;
\mathcal{U}_{T_{\ell} K^*} \in \{\mathcal{T}, T^{\ast}, L^{\ast}, R^{\ast}\} \\[0.2cm]
&\;\Longleftrightarrow\; \mathcal{V}_{T_{\ell} K^*} \in \{\mathcal{T}, T^{\ast},
L^{\ast}, R^{\ast}\} \\[0.2cm]
&\;\Longleftrightarrow\; \mathcal{V}_{T_{\ell}} \in
\mathcal{L}(\mathcal{B})\,.
\end{align*}
\vskip0.1cm

\noindent
Consequently, for $\mathcal{U}_{T_{\ell}} \in \mathcal{L}(\mathcal{B}),$
we have $\mathcal{V}_{T_{\ell}} \in \mathcal{L}(\mathcal{B})$ and
$\mathcal{U}_{T_{\ell} K} = \mathcal{T} = \mathcal{V}_{T_{\ell} K}.$ Thus
$$
\emptyset \chi_{_{\scriptstyle\mathcal{U}_{T_{\ell}}}} = \mathcal{U}_{T_{\ell}K} = \mathcal{V}_{T_{\ell}K}
= \emptyset
\chi_{_{\scriptstyle\mathcal{V}_{T_{\ell}}}}
\;\quad \mbox{if} \;\quad \mathcal{U}_{T_{\ell}} \in
\mathcal{L}(\mathcal{B})\,.
$$
On the other hand,
$$
\mathcal{U}_{T_{\ell}} \not\in \mathcal{L}(\mathcal{B}) \;\Longrightarrow\;
\mathcal{V}_{T_{\ell}} \not\in \mathcal{L}(\mathcal{B})
$$
so that
$$
\emptyset \chi_{_{\scriptstyle\mathcal{U}_{T_{\ell}}}} = \mathcal{U}_{T_{\ell} K} =
\mathcal{U}_{T_{\ell} K^*} = T_r \chi_{_{\scriptstyle\mathcal{U}}} = T_{\ell}
\chi_{_{\scriptstyle\mathcal{V}}}
= \mathcal{V}_{T_{\ell} K^*} = \mathcal{V}_{T_{\ell} K} = \emptyset
\chi_{_{\scriptstyle\mathcal{V}_{T_{\ell}}}}\vspace*{0.1cm}\,.
$$
Thus, in all cases, we have $\emptyset
\chi_{_{\scriptstyle\mathcal{U}_{T_{\ell}}}} =
\emptyset \chi_{_{\scriptstyle\mathcal{V}_{T_{\ell}}}}$ so that $\tau
\chi_{_{\scriptstyle\mathcal{U}_{T_{\ell}}}} = \tau
\chi_{_{\scriptstyle\mathcal{V}_{T_{\ell}}}}$ for all
$\tau \in \Theta^1.$ Therefore, by Theorem 4.5,
$\mathcal{U}_{T_{\ell}} = \mathcal{V}_{T_{\ell}}$ so that
$\mathcal{U}\;T_{\ell}\:\mathcal{V}.$ 


\end{proof}

Lemma 4.7 could be described as a {\em the fundamental triviality} in the study of $K$-classes.  It 
is a simple consequence from Theorem 4.5, but it is fundamental to the multiple consequences of 
Theorem 4.5 concerning $K$-classes.  This is due to the fact that if we wish to know about the 
structure of the $K$-class of $\mathcal{V} \in [\mathcal{S, CR}]$ then one powerful approach is to 
investigate those $\varphi \in \Phi$ such that $\emptyset\varphi = \emptyset\chi{_{_{\mathcal{V}}}} = 
\mathcal{V}_K$.\\

A useful way to conceptualize Lemma 4.7(ii) is to see that two varieties are $T_{\ell}$ related 
if and only if the two decreasing sequences
\begin{eqnarray*}
T_{\ell}\chi_{\mathcal{U}}, &T_{\ell}T_r\chi_{\mathcal{U}},&T_{\ell}T_rT_{\ell}\chi_{\mathcal{U}}, \ldots\\
T_{\ell}\chi_{\mathcal{V}}, &T_{\ell}T_r\chi_{\mathcal{V}},&T_{\ell}T_rT_{\ell}\chi_{\mathcal{V}}, \ldots
\end{eqnarray*}
are identical.  We refer to these chains as the $T_{\ell}$-{\em chains of} $\mathcal{U}$ and $\mathcal{V}$, 
respectively.  The $T_r$-chains are defined dually.\\

\section{Bands}

The lattice $\mathcal{L}(\mathcal{CR})$  is naturally partitioned into two classes consisting of 
those that contain the variety of all bands $\mathcal{B}$ and those that don't.  This division 
is reflected in the ladder representation of a variety.

\begin{lemma}
Let $\mathcal{V} \in [\mathcal{S}, \mathcal{CR}]$ and $\mathcal{P} \in \{\mathcal{S, LNB, RNB}\}$.  \\
{\rm (i)}
\begin{eqnarray*}
\mathcal{B} \subseteq \mathcal{V} &\Longleftrightarrow& \tau\chi_{\mathcal{V}} \neq T^{\ast} 
\;\;\mbox{for all}\;\; \tau \in \Theta^1\\
&\Longleftrightarrow& \tau\chi_{\mathcal{V}} \notin \{T^{\ast}, L^{\ast}, R^{\ast}\}
\;\;\mbox{for all}\;\; \tau \in \Theta^1\\
&\Longleftrightarrow& \mathcal{V}_{\tau} \notin \{\mathcal{S, LNB, RNB}\}\;\mbox{for all}\; 
\tau \in \Theta.
\end{eqnarray*}
{\rm (ii)} Let $\tau \in \Theta$.  Then 
$$
\mathcal{V}_{\tau} \cap \mathcal{B} = \mathcal{P} \Longleftrightarrow \mathcal{V}_{\tau} = \mathcal{P}.
$$ 
\end{lemma}
\begin{proof}
(i) First assume that $\mathcal{B} \subseteq \mathcal{V}$.  By (P1), $\emptyset \chi_{\mathcal{V}} 
\in \mathcal{K}_0$ and therefore $\emptyset \chi_{_{\mathcal{V}}}  \neq T^{\ast}$.   For any 
$\tau \in \Theta$, we have by (2) that
$$
\tau \chi_{_{\mathcal{V}}} = \mathcal{V}_{\tau K^{\ast}} \supseteq \mathcal{B}_{\tau K^{\ast}} = 
\mathcal{B}_{K^{\ast}} = \mathcal{B}_{K} = \mathcal{T} \neq T^{\ast}.
$$

Conversely, assume that $\tau\chi_{\mathcal{V}} \neq T^{\ast}$,  for all $\tau \in \Theta^1$ .  Then, 
necessarily, $\tau\chi_{\mathcal{V}} \notin \{T^{\ast}, L^{\ast}, R^{\ast}\}$, since otherwise we 
would have $\sigma \chi_{\mathcal{V}} = T^{\ast}$ for all $\sigma \in \Theta$ with $|\sigma|  
\geq |\tau| +2$.  Consequently, for all $\tau \in \Theta^1$, 
$$
\tau\chi_{\mathcal{V}} \geq \mathcal{T} = \tau\chi_{\mathcal{B}}
$$
so that $\chi_{\mathcal{V}} \geq \chi_{\mathcal{B}}$ which, since $\chi$ is an order-isomorphism, 
implies that $\mathcal{V} \supseteq \mathcal{B}$.  This establishes the first equivalence.  The 
proof of the second equivalence follows similarly.\\

For the final equivalence, since $\mathcal{S}\subseteq \mathcal{V}$, we have $\mathcal{V}_{\tau} 
\supseteq \mathcal{S}$ for all $\tau \in \Theta$ and so
$$
\tau \chi_{_{\mathcal{V}}} = T^{\ast} \Longleftrightarrow \mathcal{V}_{\tau K^{\ast}} = T^{\ast} 
\Longleftrightarrow  \mathcal{V}_{\tau} = \mathcal{S}.
$$
In like manner, 
$$
\tau \chi_{_{\mathcal{V}}} = L^{\ast} \Longleftrightarrow \mathcal{V}_{\tau} = \mathcal{LNB},\; 
\mbox{and}\; 
\tau \chi_{_{\mathcal{V}}} = R^{\ast} \Longleftrightarrow \mathcal{V}_{\tau} = \mathcal{RNB}.
$$
This establishes the third equivalence.\\

\noindent
(ii)  Let $\mathcal{P} \in \{\mathcal{T, LNB, RNB}\}$.  Trivially 
$$
\mathcal{V}_{\tau} = \mathcal{P} \Longrightarrow \mathcal{V}_{\tau} \cap \mathcal{B} = \mathcal{P}.
$$
Conversely, assume that $\mathcal{V}_{\tau} \cap \mathcal{B} = \mathcal{P}$.  First consider the 
case where $\mathcal{P} = \mathcal{S}$.  In this case, it follows that $\mathcal{V}_{\tau}$ contains 
no non-trivial left zero semigroups or right zero semigroups.  Hence $\mathcal{V}$ is a variety of 
semilattices of groups.  Without loss of generality, we may assume that $t(\tau) = T_{\ell}$.  But that 
means that $\mathcal{V}_{\tau}$ is the least element in its $T_{\ell}$-class.  Since $\mathcal{SG}_{T_{\ell}} 
= \mathcal{S}$ we must have $\mathcal{V}_{\tau} =\mathcal{S}$.  Now consider the case where 
$\mathcal{P} = \mathcal{LNB}$.  

In this case, it follows that $\mathcal{V}_{\tau}$ contains 
no non-trivial right zero semigroups.  Hence $\mathcal{V}$ is a variety of 
semilattices of left groups.  Let $S \in \mathcal{V}$.  Then, as a semilattice of left groups, 
$S$ is orthodox and the subsemigroup of idempotents $E(S)$ must lie in $\mathcal{LNB}$.  
This implies that $S \in \mathcal{LNO}$, the variety of left normal orthogroups.  Hence 
$\mathcal{V} \subseteq \mathcal{LNO}$.   In this case, if 
 $t(\tau) = T_{\ell}$  
then $\mathcal{V}_{\tau}$ is again the least element in its $T_{\ell}$-class.  But $\mathcal{LNO}_{T_{\ell}} 
= \mathcal{S}$ so this would imply that $\mathcal{V}_{\tau} =\mathcal{S}$ which would contradict 
the assumption 
that $\mathcal{V}_{\tau} \cap \mathcal{B} = \mathcal{LNB}$. Hence, we must have 
$t(\tau) = T_r$ and $\mathcal{V}_{\tau}$ is the least element in its $T_r$-class.   
Now we have $\mathcal{LNB} \subseteq \mathcal{V}_{\tau}\subseteq \mathcal{LNO}$ where 
$\mathcal{LNO}_{T_r} = \mathcal{LNB}$.  Consequently, we must have 
$\mathcal{V}_{\tau} = \mathcal{LNO}$.  A dual argument will deal with the case 
$\mathcal{P} = \mathcal{RNB}$ and the proof is complete.
\end{proof}

\noindent
The following mappings provide a basic tool for the study of $\mathcal{L}(\mathcal{CR})$ in 
general and, in particular, for any $K$-class in $\mathcal{L}(\mathcal{CR})$:

\begin{eqnarray*}
\mu_{\mathcal{B}} &:& \mathcal{V} \longrightarrow \mathcal{V} \cap \mathcal{B}\\
\theta &:& \mathcal{V} \longrightarrow (\mathcal{V}\vee \mathcal{B}, \mathcal{V}\cap \mathcal{B}).
\end{eqnarray*}

Following [Pe2007] we denote by ${\mathbf{B}}$ the relation induced on 
$\mathcal{L}(\mathcal{CR})$ by the 
mapping $\mu_{\mathcal{B}}$.

\begin{theorem} 
{\rm (i)} The mapping $\mu_{\mathcal{B}}$ is a complete retraction of $\mathcal{L}(\mathcal{CR})$ 
onto $\mathcal{L}(\mathcal{B})$.  The mapping $\theta$ is a faithful representation of 
$\mathcal{L}(\mathcal{CR})$ as a subdirect product of $[\mathcal{B}, \mathcal{CR}] 
\times \mathcal{L}(\mathcal{B})$.  In other words, $\mathcal{B}$ is a neutral element in 
$\mathcal{L}(\mathcal{CR})$.\\
{\rm (ii)}  The relation ${\mathbf{B}}$ induced on $\mathcal{L}(\mathcal{CR})$ by $\mu_{\mathcal{B}}$ 
is a complete congruence and so its classes 
are intervals of the form $[\mathcal{V}_{\mathbf{B}}, \mathcal{V}^{\mathbf{B}}]$ where 
$\mathcal{V} \in \mathcal{L}(\mathcal{CR})$ and $\mathcal{V}_{\mathbf{B}} = 
\mathcal{V} \cap \mathcal{B}$.\\
{\rm (iii)}  Let $\mathcal{V}_{\alpha} \in \mathcal{L}(\mathcal{CR}), \alpha \in A$.  Then $(\bigcap\limits_{\alpha \in A} \mathcal{V}_{\alpha})^{{\bf B}} = 
\bigcap\limits_{\alpha \in A} \mathcal{V}_{\alpha}^{{\bf B}}$.\\
{\rm (iv)}  Let $\mathcal{V} \in \{\mathcal{S, LNB, RNB}\}$ and $u \in \Theta^1$.  Then 
$(\mathcal{V}^{u(K_{\ell}, K_r)})^{{\bf B}} = \mathcal{V}^u$.\\
{\rm (v)}  Let $\mathcal{V} \in \mathcal{L}(\mathcal{CR}), u \in \Theta^1$.  Then 
$\mathcal{V}^u\cap \mathcal{B} = (\mathcal{V}\cap \mathcal{B})^{u(K_{\ell}, K_r)}$.\\
{\rm (vi)} There is an algorithm to generate (finite) bases of identities for varieties of the form 
$\mathcal{U}^{\bf B}$ 
for $\mathcal{U} \in \mathcal{L}(\mathcal{B})$.
\end{theorem} 
\begin{proof}
For parts (i) and (ii), see (Trotter [T]).\\
(iii)  See Reilly/Zhang [RZ2000], Lemma 4.9.\\
(iv)  See Reilly/Zhang [RZ2000], Lemma 4.11.\\
(v)  Recall that we denote by $\tau$ the greatest idempotent pure congruence on a 
completely regular semigroup.    We always have $(\tau\cap \mathcal{L})^0 \subseteq 
\mathcal{L}^0$.  However, in a band, the congruence $\mathcal{L}^0$ is idempotent pure 
so that $\mathcal{L}^0 \subseteq \tau\cap\mathcal{L}$.  Hence $\mathcal{L}^0 \subseteq 
(\tau\cap \mathcal{L})^0$.  Therefore, in any band, $\mathcal{L}^0 = (\tau\cap \mathcal{L})^0$.  
Then, we have
\begin{eqnarray*}
(\mathcal{V}^{T_{\ell}}) \cap \mathcal{B} &=& \{S \in \mathcal{CR}| S/\mathcal{L}^0 \in \mathcal{V}\}\cap 
\mathcal{B}\;\;\mbox{by Lemma 2.1(iv)}\\
&=& \{S \in \mathcal{B}| S/\mathcal{L}^0 \in \mathcal{V}\}\\
&=& \{S \in \mathcal{B}| S/\mathcal{L}^0 \in \mathcal{V}\cap \mathcal{B}\}\\
&=& \{S \in \mathcal{B}| S/(\tau \cap\mathcal{L})^0 \in \mathcal{V}\cap \mathcal{B}\}\\
&=& \{S \in \mathcal{CR}| S/(\tau \cap\mathcal{L})^0 \in \mathcal{V}\cap \mathcal{B}\}\\
&=& (\mathcal{V} \cap \mathcal{B})^{K_{\ell}}\;\;\mbox{by Lemma 2.1(iv)}.
\end{eqnarray*}
Dually, we also have $(\mathcal{V}^{T_r}) \cap \mathcal{B} = (\mathcal{V} \cap \mathcal{B})^{K_{r}}$.  
The claim now follows by a simple induction argument. \\
(vi)  See [RZ2000] and [Pe2007].
\end{proof}

Reilly and Zhang [RZ2000] provide an inductive process to generate bases of
identities for the varieties of the form $\mathcal{V}^B$
$\;(\mathcal{V} \in \mathcal{L}(\mathcal{B}))$ based on the bases
of identities for varieties in $\mathcal{L}(\mathcal{B})$ developed
by Gerhard and Petrich [GP]. Similar results were obtained  
by Petrich [Pe2007].   \\

\noindent
Although the relation ${\bf B}$ does not appear explicitly in the definition of the mapping $\xi$ in 
Pol\'{a}k's Theorem, the next result shows how the two concepts are related.

\begin{theorem}
Let $\mathcal{V} \in [\mathcal{S, CR}]$.  Then
\[\mathcal{V} = 
\left \{ \begin{array}{ll}
\bigcap\limits_{\tau \in \Theta^1} (\mathcal{V}_{\tau})^{K\overline{\tau}} & \mbox{if $\mathcal{B} \subseteq 
\mathcal{V}$}\\

\mathcal{V}^K \cap \bigcap\limits_{\mathcal{V}_{\tau K^{\ast}} \in \mathcal{K}_0} 
(\mathcal{V}_{\tau}^{K\overline{\tau}}) \cap 
\mathcal{V}^{{\bf B}} & \mbox{otherwise.}
\end{array}
\right. \]
\end{theorem}
\begin{proof}
From Theorem 4.5, we know that 
$\mathcal{V} = \mathcal{V}\chi\xi = \chi_{_{\mathcal{V}}}\xi$ which, when written  in component  
form, provides
$$
\mathcal{V} = \chi_{_{\mathcal{V}}} \xi = (\emptyset \chi_{_{\mathcal{V}}} )^K \cap \bigcap\limits_{\tau 
\chi_{_{\mathcal{V}}}
\in \mathcal{K}_0} (\tau \chi_{_{\mathcal{V}}} )^{K\overline{\tau}} 
\cap  \bigcap\limits_{\tau \chi_{_{\mathcal{V}}} = T^*}
\mathcal{S}^{\overline{\tau}}
\cap
\bigcap\limits_{\tau \chi_{_{\mathcal{V}}} = L^*} \mathcal{LNB}^{\overline{\tau}}
\cap \bigcap\limits_{\tau \chi_{_{\mathcal{V}}} = R^*} \mathcal{RNB}^{\overline{\tau}}.
$$
First,  $(\emptyset \chi_{_{\mathcal{V}}} )^K = (\mathcal{V}_K)^K = 
\mathcal{V}^K$.  Now suppose that $\mathcal{B} \subseteq \mathcal{V}$.  By Lemma 5.1(i), 
the set of $\tau \in \Phi$ such that $\tau\chi_{_{\mathcal{V}}} = T^{\ast}$ is empty.  Thus 
there is no third component in the component form of $\mathcal{V}$.  
Likewise, there is no fourth or fifth 
component.  Thus we are left with 
$$
\mathcal{V} = \chi_{\mathcal{V}}\xi = \mathcal{V}^K \cap \bigcap\limits_{\tau \in \Theta} (\tau \chi_{_{\mathcal{V}}} )^{K\overline{\tau}} = \mathcal{V}^K \cap \bigcap\limits_{\tau \in \Theta} (\mathcal{V}_{\tau K} )^{K\overline{\tau}} 
= \bigcap\limits_{\tau \in \Theta^1} (\mathcal{V}_{\tau} )^{K\overline{\tau}} 
$$ 
as required.\\

Now suppose that $\mathcal{B}\nsubseteq \mathcal{V}$.  By Lemma 5.1, there exists $\tau \in \Theta$ 
such that $\tau\chi_{_{\mathcal{V}}} \in \{T^{\ast}, L^{\ast}, R^{\ast}\}$.  
The first thing that we want to do is 
obtain an expression for $\mathcal{V}\cap \mathcal{B}$.
As before, the first component of $\chi_{_{\mathcal{V}}}\xi$ 
is simply $\mathcal{V}^K$ which contains $\mathcal{B}$.    Each term in the second component 
is of the form $(\tau \chi_{_{\mathcal{V}}} )^{K\overline{\tau}}$ where $\tau \chi_{_{\mathcal{V}}} 
\in \mathcal{K}_0$ and therefore $\tau \chi_{_{\mathcal{V}}} \supseteq \mathcal{T}$.  Hence 
$(\tau \chi_{_{\mathcal{V}}} )^{K\overline{\tau}}\supset \mathcal{T}^{K\overline{\tau}} 
= \mathcal{B}^{\overline{\tau}} 
\supseteq \mathcal{B}$.  Therefore we obtain, with the help of Lemma 5.2(v) in the third step below, 

\begin{eqnarray*}
\mathcal{V} \cap \mathcal{B} &=& \bigcap\limits_{\tau \chi_{_{\mathcal{V}}} = T^*}
\mathcal{S}^{\overline{\tau}}
\cap
\bigcap\limits_{\tau \chi_{_{\mathcal{V}}} = L^*} \mathcal{LNB}^{\overline{\tau}}
\cap \bigcap\limits_{\tau \chi_{_{\mathcal{V}}} = R^*} \mathcal{RNB}^{\overline{\tau}} \;
\cap \mathcal{B}\\
&=& \bigcap\limits_{\tau \chi_{_{\mathcal{V}}} = T^*}
(\mathcal{S}^{\overline{\tau}}\cap \mathcal{B})
\cap
\bigcap\limits_{\tau \chi_{_{\mathcal{V}}} = L^*} (\mathcal{LNB}^{\overline{\tau}}\cap \mathcal{B} 
\cap \bigcap\limits_{\tau \chi_{_{\mathcal{V}}} = R^*} (\mathcal{RNB}^{\overline{\tau}} 
\cap \mathcal{B})\\
&=&
\bigcap\limits_{\tau \chi_{_{\mathcal{V}}} = T^*}
\mathcal{S}^{\overline{\tau (K_{\ell}, K_r)}}
\cap
\bigcap\limits_{\tau \chi_{_{\mathcal{V}}} = L^*} \mathcal{LNB}^{\overline{\tau (K_{\ell}, K_r)}}
\cap \bigcap\limits_{\tau \chi_{_{\mathcal{V}}} = R^*} \mathcal{RNB}^{\overline{\tau (K_{\ell}, K_r)}}.
\end{eqnarray*}
Consequently, by Lemma 5.2(iii) and (iv),
\begin{eqnarray*}
\mathcal{V}^{{\bf B}} &=& 
(\mathcal{V} \cap \mathcal{B})^{{\bf B}}\\ 
&=&
(\bigcap\limits_{\tau \chi_{_{\mathcal{V}}} = T^*}
\mathcal{S}^{\overline{\tau (K_{\ell}, K_r)}}
\cap
\bigcap\limits_{\tau \chi_{_{\mathcal{V}}} = L^*} \mathcal{LNB}^{\overline{\tau (K_{\ell}, K_r)}}
\cap \bigcap\limits_{\tau \chi_{_{\mathcal{V}}} = R^*} \mathcal{RNB}^{\overline{\tau (K_{\ell}, K_r)}})
^{{\bf B}}.\\
&=&
\bigcap\limits_{\tau \chi_{_{\mathcal{V}}} = T^*}
(\mathcal{S}^{\overline{\tau (K_{\ell}, K_r)}})^{{\bf B}}
\cap
\bigcap\limits_{\tau \chi_{_{\mathcal{V}}} = L^*} (\mathcal{LNB}^{\overline{\tau (K_{\ell}, K_r)}})^{{\bf B}}
\cap \bigcap\limits_{\tau \chi_{_{\mathcal{V}}} = R^*} (\mathcal{RNB}^{\overline{\tau (K_{\ell}, K_r)}})
^{{\bf B}}\\
&=& \bigcap\limits_{\tau \chi_{_{\mathcal{V}}} = T^*}
\mathcal{S}^{\overline{\tau}}
\cap
\bigcap\limits_{\tau \chi_{_{\mathcal{V}}} = L^*} \mathcal{LNB}^{\overline{\tau}}
\cap \bigcap\limits_{\tau \chi_{_{\mathcal{V}}} = R^*} \mathcal{RNB}^{\overline{\tau}}.
\end{eqnarray*}
In other words, the last three components in the expression $\chi_{\mathcal{V}}$ for $\mathcal{V}$ 
combine 
to give exactly $\mathcal{V}^{{\bf B}}$.  Thus the long form expression for $\mathcal{V}$ reduces in 
this case to 
precisely
$$
\mathcal{V}^K \cap \bigcap\limits_{\mathcal{V}_{\tau K^{\ast}} \in \mathcal{K}_0} 
(\mathcal{V}_{\tau}^{K\overline{\tau}}) \cap 
\mathcal{V}^{{\bf B}} 
$$
as required.
\end{proof}

In the proof of Theorem 5.3, we obtained an expression for $\mathcal{V} \cap \mathcal{B}$ in terms 
of the last three components of the component form for $\mathcal{V}$.  However, $\mathcal{S} = 
\mathcal{LRB}\cap \mathcal{RRB} = \mathcal{LNB}^{K_{\ell}} \cap \mathcal{RNB}^{K_r}$ so that, in 
fact, every variety of bands can be expressed as the intersection of varieties of the form 
$\mathcal{LNB}^{\overline{\tau (K_{\ell}, K_r)}}$ and $\mathcal{RNB}^{\overline{\tau (K_{\ell}, K_r)}}$ 
(see Pastijn [J.Austral. 1990] for details).  However, we did not require that refinement for this 
proof.  The proof also provides us with a nice expression for the variety $\mathcal{V}^{{\bf B}}$ 
$(\mathcal{V} \in \mathcal{L}(\mathcal{CR}))$
which we can simplify slightly as follows.

\begin{corollary}  Let $\mathcal{V} \in [\mathcal{S, CR}]$.  Then
$$
\mathcal{V}^{{\bf B}} = \mathcal{CR}\cap \bigcap\limits_{\mathcal{V}_{\tau} = \mathcal{S}}
\mathcal{S}^{\overline{\tau}}
\cap
\bigcap\limits_{\mathcal{V}_{\tau} = \mathcal{LNB}} \mathcal{LNB}^{\overline{\tau}}
\cap \bigcap\limits_{\mathcal{V}_{\tau} = \mathcal{RNB}} \mathcal{RNB}^{\overline{\tau}}.
$$
\end{corollary}
\begin{proof}
If $\mathcal{B}\subseteq \mathcal{V}$, then clearly $\mathcal{V}^{{\bf B}} = \mathcal{CR}$.  On 
the other hand by Lemma 5.1(i), the second, third and fourth components of the expression 
on the right in the statement of the corollary are vacuous and therefore the right hand side 
reduces to $\mathcal{CR}$.  \\

So suppose that $\mathcal{B}\nsubseteq \mathcal{V}$.  From the proof of Theorem 5.3, we 
have:
$$
\mathcal{V}^{{\bf B}} = \bigcap\limits_{\tau \chi_{_{\mathcal{V}}} = T^*}
\mathcal{S}^{\overline{\tau}}
\cap
\bigcap\limits_{\tau \chi_{_{\mathcal{V}}} = L^*} \mathcal{LNB}^{\overline{\tau}}
\cap \bigcap\limits_{\tau \chi_{_{\mathcal{V}}} = R^*} \mathcal{RNB}^{\overline{\tau}}.
$$
Let 
$\tau \in \Theta$ be such that $\tau\chi_{_{\mathcal{V}}} = T^{\ast}$.  Then 
$\mathcal{V}_{\tau K^{\ast}} = T^{\ast}$.  Since $\mathcal{V}_{\tau} \supseteq \mathcal{S}$, it 
follows that $\mathcal{V}_{\tau} = \mathcal{S}$.  Conversely, if $\mathcal{V}_{\tau} = \mathcal{S}$, 
then $\mathcal{V}_{\tau K^{\ast}} = T^{\ast}$.  Consequently, the component 
$\bigcap\limits_{\tau \chi_{_{\mathcal{V}}} = T^*}\mathcal{S}^{\overline{\tau}}$ 
 can be replaced by $\bigcap\limits_{\mathcal{V}_{\tau} = \mathcal{S}}
\mathcal{S}^{\overline{\tau}}$.  In like manner, we can replace 
$\bigcap\limits_{\tau \chi_{_{\mathcal{V}}} = L^*} \mathcal{LNB}^{\overline{\tau}}$ by
$\bigcap\limits_{\mathcal{V}_{\tau} = \mathcal{LNB}} \mathcal{LNB}^{\overline{\tau}}$ and 
$\bigcap\limits_{\tau \chi_{_{\mathcal{V}}} = R^*} \mathcal{RNB}^{\overline{\tau}}$ by 
$\bigcap\limits_{\mathcal{V}_{\tau} = \mathcal{RNB}} \mathcal{RNB}^{\overline{\tau}}$ 
thereby obtaining the claim.
\end{proof}

In the final result of this section, we take advantage of Corollary 5.4 to derive a characterization 
of the ${\bf B}$-relation in a manner related to ladders.

\begin{corollary}
Let $\mathcal{U, V}\in [\mathcal{S, CR}]$.  Then the following statements are equivalent.\\
{\rm (i)}  $\mathcal{U}\: {\bf B}\: \mathcal{V}$.\\
{\rm (ii)} For all $\mathcal{P} \in \{\mathcal{S, LNB, RNB}\}, \tau \in \Theta$, we have 
$$
\mathcal{U}_{\tau} = \mathcal{P} \Longleftrightarrow \mathcal{V}_{\tau} = \mathcal{P}. 
$$
{\rm (iii)}  For all $\mathcal{P} \in \{T^{\ast}, L^{\ast}, R^{\ast}\}, \tau \in \Theta$, we have 
$$
\tau\chi_{_{\mathcal{U}}} = \mathcal{P} \Longleftrightarrow \tau\chi_{_{\mathcal{V}}} = \mathcal{P}. 
$$
\end{corollary}
\begin{proof} (i) $\Longleftrightarrow$ (ii).  
First assume that $\mathcal{U}\; {\bf B}\; \mathcal{V}$ and let $\tau \in \Theta$ be such 
that $\mathcal{U}_{\tau} = \mathcal{P} \in \{\mathcal{S, LNB,} \mathcal{RNB}\}$.  Then
$$
\mathcal{V}_{\tau}\cap \mathcal{B} = (\mathcal{V}\cap \mathcal{B})_{\tau} = 
(\mathcal{U}\cap \mathcal{B})_{\tau} = \mathcal{U}_{\tau}\cap \mathcal{B} 
= \mathcal{U}_{\tau} =\mathcal{P} \in \{\mathcal{S, LNB, RNB}\}.
$$
By Lemma 4.1, we have $\mathcal{V}_{\tau} = \mathcal{P} = \mathcal{U}_{\tau}$, as required.  
Now suppose that $\mathcal{U, V}$ are related as in part (ii).   From Corollary 5.4, we have
\begin{eqnarray*}
\mathcal{U}^{{\bf B}} &=& \mathcal{CR}\cap \bigcap\limits_{\mathcal{U}_{\tau} = \mathcal{S}}
\mathcal{S}^{\overline{\tau}}
\cap
\bigcap\limits_{\mathcal{U}_{\tau} = \mathcal{LNB}} \mathcal{LNB}^{\overline{\tau}}
\cap \bigcap\limits_{\mathcal{U}_{\tau} = \mathcal{RNB}} \mathcal{RNB}^{\overline{\tau}}\\
&=& \mathcal{CR}\cap \bigcap\limits_{\mathcal{V}_{\tau} = \mathcal{S}}
\mathcal{S}^{\overline{\tau}}
\cap
\bigcap\limits_{\mathcal{V}_{\tau} = \mathcal{LNB}} \mathcal{LNB}^{\overline{\tau}}
\cap \bigcap\limits_{\mathcal{V}_{\tau} = \mathcal{RNB}} \mathcal{RNB}^{\overline{\tau}}\\
&=& \mathcal{V}^{{\bf B}}.
\end{eqnarray*}
Therefore $\mathcal{U}\;{\bf B}\;\mathcal{V}$.\\

\noindent
(iii)  Recall that $\tau\chi_{_{\mathcal{U}}} = \mathcal{U}_{\tau K^{\ast}}$ and 
$\tau\chi_{_{\mathcal{V}}} = \mathcal{V}_{\tau K^{\ast}}$.  Part (iii) is then a simple 
reformulation of (ii).
\end{proof}


\section{Ladders in $[\mathcal{S}, L\mathcal{O}]$}.

Our next goal is to study the nature of some $K$-classes in $\mathcal{L}(\mathcal{O})$ and 
$\mathcal{L}(L\mathcal{O})$ in some 
detail.  For any $\mathcal{V} \in \mathcal{L}(\mathcal{O})$ we have that 
$\emptyset\chi_{\mathcal{V}} = \mathcal{V}_K \subseteq \mathcal{G}$.  So our interest is 
focussed on those elements $\varphi \in \Phi$ for which $\tau\varphi \leq \mathcal{G}$ for 
all $\tau \in \Theta^1$ and, for $\mathcal{L}(L\mathcal{O})$, our interest is in those $\varphi 
\in \Phi$ for which $\tau\varphi \leq \mathcal{CS}$ for all $\tau \in \Theta^1$.    
Pol\'{a}k ([Po3], Theorem 4.2 (2)) considers 
the intervals $[\mathcal{S}, \mathcal{O}]$ and $[\mathcal{S}, L\mathcal{O}]$.\\  

\noindent
For any $\mathcal{V} \in \mathcal{L}(L\mathcal{O})$, we have $\mathcal{V}_K \in \mathcal{L}(\mathcal{G}) 
\cup (\mathcal{L}(\mathcal{CS})\backslash \mathcal{L}(\mathcal{R}e\mathcal{G}))$ so that 
$\mathcal{V}K\cap \mathcal{K}_0 = \{\mathcal{V}_K\}$ - a singleton set.  Therefore the 
interesting part of $\mathcal{V}K$ lies above $\mathcal{S}$ and so we concentrate our 
attention on the intervals $[\mathcal{S, O}]$ and $[\mathcal{S}, L\mathcal{O}]$.  Furthermore, 
in the spirit of Theorem 5.2 (i), (ii), it seems natural to break these intervals into two 
sublattices, namely those varieties containing $\mathcal{B}$ and those that do not.    \\

Let $\overline{[\mathcal{B, O}]} = [\mathcal{S}, \mathcal{O}]\backslash [\mathcal{B}, \mathcal{O}]$
denote the {\em complement} of the interval $[\mathcal{B, O}]$ in $[\mathcal{S}, \mathcal{O}]$.  \\

Note that $\mathcal{L}(\mathcal{G}) \cup \mathbb{N}_{\,3}^*$ is a sublattice 
(and ideal) in $\mathcal{K}$. \\

\begin{theorem}
{\rm (i)} $\Phi_{[\mathcal{B, O}]}$ consists of all order preserving mappings of $\Theta^1$ into 
$\mathcal{L}(\mathcal{G})$.\\
{\rm (ii)} $\Phi_{[\mathcal{B, O}]}$ is an interval in $\Phi$.\\
{\rm (iii)} $\Phi_{\overline{[\mathcal{B, O}]}}$ consists of all order preserving mappings $\varphi$ 
of $\Theta^1$ into $\mathcal{L}(\mathcal{G})\cup {\mathbf{N}}_3^{\ast}$ satisfying (P1)-(P4) 
and the additional condition:\\

$T^{\ast}: \tau\varphi = T^{\ast}$ for some $\tau \in \Theta$.\\

\noindent
{\rm (iv)}  $\Phi_{\overline{[\mathcal{B, O}]}}$ is a sublattice but  not a complete sublattice 
of $\Phi$.
\end{theorem}

{\bf Note:} Thus elements of $\Phi_{[\mathcal{B, O}]}$ never assume the value $T^{\ast}$ 
while for any element 
$\varphi$ of $\Phi_{\overline{[\mathcal{B, O}]}}$ there is some $\tau \in \Theta$ such that 
$\tau\varphi = T^{\ast}$ and therefore, since $\varphi$ is order-preserving, there is some 
$\sigma \in \Theta$ such that $\tau\varphi = T^{\ast}$ for all $\tau \leq \sigma$.  \\
\begin{proof}  (i)  Let $\mathcal{P}$ denote the set of 
all order preserving mappings of $\Theta^1$ into 
$\mathcal{L}(\mathcal{G})$.\\

\noindent
 Let $\mathcal{V} \in [\mathcal{B, O}]$.  By Lemma 5.1, it follows that $\tau\chi_{\mathcal{V}} 
\supseteq \mathcal{T}$ for all $\tau$.  Now $\tau\chi_{\mathcal{O}} = \mathcal{G}$ for all 
$\tau \in \Theta^1$.   By (P2), $\chi$ is order preserving and so, for all $\tau \in \Theta^1$, we have  $\mathcal{T} \subseteq 
\tau\chi_{\mathcal{V}} \subseteq \tau\chi_{\mathcal{O}} = \mathcal{G}$ so that 
$\tau\chi_{\mathcal{V}} \in \mathcal{L}(\mathcal{G})$.  Since  
$\chi_{\mathcal{V}}$ is order-preserving we have $\chi_{\mathcal{V}} \in \mathcal{P}$.  \\

\noindent
For the converse, we must first show that $\mathcal{P}\subseteq \Phi$.  
Let $\varphi \in \mathcal{P}$.  Since $\mathcal{L}(\mathcal{G}) \subseteq \mathcal{K}_0$, 
it is clear that  
$\varphi$ satisfies (P1) and (P2).  Conditions (P3) and (P4) are satisfied vacuously since 
$\tau\varphi \in \mathcal{L}(\mathcal{G})$ for all $\tau$.  We have $\emptyset\varphi \in 
\mathcal{L}(\mathcal{G})$.  Hence, for any $\tau \in \Theta$, 
$$
(\emptyset \varphi)_{\tau K^{\ast}} = \mathcal{T}_{K^{\ast}} = T^{\ast} \subseteq \tau\varphi 
$$
so that $\varphi$ satisfies (P5).   Next let $\sigma ,\tau \in \Theta$ be such that $\sigma
\varphi \in \mathcal{K}_0$ (which is no restrction in this context) 
and $t(\sigma ) \neq h(\tau )$.  Then $\sigma \varphi \in 
\mathcal{L}(\mathcal{G})$ so that $(\sigma \varphi)_{\tau} = \mathcal{T}$.  Hence 
$(\sigma \varphi)_{\tau K^{\ast}} = T^{\ast} \subseteq (\sigma\tau)\varphi$.   Thus $\varphi$ 
satisfies (P6) so that $\varphi \in \Phi$ and $\mathcal{P} \subseteq \Phi.$ \\

\noindent
Conversely, let $\varphi \in \mathcal{P}$.  Since $\mathcal{P} \subseteq 
\Phi$, it follows from Theorem 4.5, that there exists a variety $\mathcal{V} \in [\mathcal{S, CR}]$ 
such that $\varphi = \chi_{\mathcal{V}}.$  By Lemma 5.1, $\mathcal{B} \subseteq \mathcal{V}$ so 
that $\mathcal{V}$ actually lies in $[\mathcal{B, CR}]$.  We also have $\mathcal{V}_K = 
\emptyset\chi_{\mathcal{V}} \in \mathcal{L}(\mathcal{G}).$   
But $\mathcal{G}^K = \mathcal{O}$ and so $\mathcal{V} \in \mathcal{L}(\mathcal{O})$.  
Therefore $\mathcal{V} \in [\mathcal{B, O}]$ as required.\\

\noindent
{\rm (ii)}  Since $\Phi_{[\mathcal{B, O}]}$ is the image of $[\mathcal{B, O}]$ under $\chi$, it 
must be an interval.\\

\noindent
{\rm (iii)}  Let $\mathcal{P}$ denote the set of all $\varphi$ as described in the statement.  
Let $\mathcal{V} \in \overline{[\mathcal{B, O}]}$.  By Theorem 4.5, $\chi_{\mathcal{V}} \in \Phi$ 
and therefore $\chi_{\mathcal{V}}$ satisfies conditions (P1) - (P4).  
By Lemma 5.1, there exists $\tau \in \Theta$ such that  $\tau\chi_{\mathcal{V}} = T^{\ast}$. 
Hence $\chi_{\mathcal{V}}$ 
satisfies (T)* and so $\chi_{\mathcal{V}} \in \Phi_{\overline{[\mathcal{B, O}]}}$.\\

\noindent
For the converse, we must first show that $\mathcal{P} \subseteq \Phi$.
Let $\varphi \in \mathcal{P}$.  Since 
$\varphi$ satisfies (P1) - (P4), by hypothesis, it remains to show that $\varphi$ satisfies 
(P5) and (P6).   However, once we recognize that we still have 
$\emptyset\varphi \in \mathcal{L}(\mathcal{G})$ and that if $\sigma\varphi \in \mathcal{K}_0$ 
we again have $\sigma\varphi \in \mathcal{L}(\mathcal{G})$, then the argument proceeds exactly as 
in part (i).   \\

\noindent
Now let $\varphi \in \mathcal{P}$.  Since 
$\Phi_{\overline{[\mathcal{B, O}]}} \subseteq 
\Phi$, it follows from Theorem 4.5, that there exists a variety $\mathcal{V} \in [\mathcal{S, CR}]$ 
such that $\varphi = \chi_{\mathcal{V}}.$  By Lemma 5.1, $\mathcal{B} \nsubseteq \mathcal{V}$.  
We also have $\mathcal{V}_K = \emptyset\chi_{\mathcal{V}} \in \mathcal{L}(\mathcal{G}).$   But $\mathcal{G}^K = \mathcal{O}$ and so $\mathcal{V} \in \mathcal{L}(\mathcal{O})$.  
Therefore $\mathcal{V} \in \overline{[\mathcal{B, O}]}$ as required.\\
\noindent
{\rm (iv)}  
Clearly, $\Phi_{\overline{[\mathcal{B, O}]}}$ is closed under finite joins and meets.  To see that $\Phi_{\overline{[\mathcal{B, O}]}}$ is not a complete lattice, however, it suffices to note that 
$\Phi_{\overline{[\mathcal{B, O}]}}$ contains every variety of bands other than the variety 
$\mathcal{B}$ of all bands and that the join of all proper subvarieties of $\mathcal{B}$ is 
$\mathcal{B}$ itself which does not lie in $\Phi_{\overline{[\mathcal{B, O}]}}$.
\end{proof} 

The next step is to extend Theorem 6.1 to the level of locally orthodox varieties.  The result is 
not quite so neat as for orthodox varieties for the following reason.  We have $\mathcal{O}_K = 
\mathcal{G} = \mathcal{O} \cap \mathcal{G}$ and, for any $\mathcal{V} \in \mathcal{L}(\mathcal{O})$ 
we have $\mathcal{V}_K = \mathcal{V}\cap \mathcal{G}$.  However, although $(L\mathcal{O})_K = 
\mathcal{CS} = L\mathcal{O} \cap \mathcal{CS}$, it is not the case that for any $\mathcal{V} \in 
\mathcal{L}(L\mathcal{O})$ we have $\mathcal{V}_K = \mathcal{V} \cap \mathcal{CS}$.  In particular, 
if $\mathcal{V}$ is a rectangular band of groups, that is $\mathcal{V} = \mathcal{U}\vee \mathcal{W}$ 
for some $\mathcal{U} \subseteq \mathcal{RB}, \mathcal{W} \subseteq \mathcal{G}$, then 
$\mathcal{V}_K = \mathcal{V}\cap \mathcal{G} = \mathcal{W}$.   So, whereas 
$\mathcal{K}_0 \cap \mathcal{L}(\mathcal{O}) = \{\mathcal{V}_K| 
\mathcal{V} \in \mathcal{L}(\mathcal{O})\} 
= \mathcal{L}(\mathcal{G})$ is a sublattice of $\mathcal{L}(\mathcal{CR})$, 
$$\mathcal{K}_0 \cap \mathcal{L}(L\mathcal{O}) = \{\mathcal{V}_K| 
\mathcal{V} \in L\mathcal{O}\} = (\mathcal{L}(\mathcal{CS})\backslash \mathcal{L}(\mathcal{R}e\mathcal{B})) 
\cup \mathcal{L}(\mathcal{G})
$$ 
which is not a sublattice of $\mathcal{L}(\mathcal{CR})$.  This is the same problem that we faced with 
the definition of the lattice structure of $\mathcal{K}$ and here we are really just dealing with an 
ideal in $\mathcal{K}$.  So, as before, $(\mathcal{L}(\mathcal{CS})\backslash \mathcal{L}(\mathcal{R}e\mathcal{B})) \cup \mathcal{L}(\mathcal{G})$ {\em is} a lattice, indeed a 
complete lattice with respect 
to the order inherited from $\mathcal{L}(\mathcal{CR})$.  Equivalently, it is a lattice with respect to 
the following lattice operations $\wedge$ and $\vee$ where $\vee$ is the join operation within 
$\mathcal{L}(\mathcal{CR})$ and
$$
\mathcal{U} \wedge \mathcal{V} = (\mathcal{U} \cap \mathcal{V})_K.
$$
This derives from the fact that we are simply picking convenient representatives from within the $K$-classes and giving 
the set of these representatives the order inherited from the quotient $\mathcal{L}(L\mathcal{O})/K$.\\

When we refer to $(\mathcal{L}(\mathcal{CS})\backslash \mathcal{L}(\mathcal{R}e\mathcal{B})) \cup \mathcal{L}(\mathcal{G})$ below in this 
section, we will always mean together with the lattice structure defined above.\\

We will also want to consider the extended lattice obtained by adjoining $ \mathbb{N}_{\,3}^*$ to 
obtain the set 
$$
\mathcal{K}_{\mathcal{CS}} = 
(\mathcal{L}(\mathcal{CS})\backslash \mathcal{L}(\mathcal{R}e\mathcal{B})) \cup \mathcal{L}(\mathcal{G}) 
\cup \mathbb{N}_{\,3}^*.
$$
It is straightforward to see that this is indeed a lattice with respect to the order relation 
inherited from $\mathcal{K}$.\\

Let $\overline{[\mathcal{B}, L\mathcal{O}]} = [\mathcal{S}, L\mathcal{O}]\backslash [\mathcal{B}, L\mathcal{O}]$
denote the {\em complement} of the interval $[\mathcal{B}, L\mathcal{O}]$ in $[\mathcal{S}, L\mathcal{O}]$.  \\

\begin{theorem}
{\rm (i)}  $\Phi_{[\mathcal{B}, L\mathcal{O}]}$ consists of all order-preserving mappings of $\Theta^1$ into \linebreak $(\mathcal{L}(\mathcal{CS})\backslash \mathcal{L}(\mathcal{R}e\mathcal{B})) \cup \mathcal{L}(\mathcal{G})$. \\
{\rm (ii)}  $\Phi_{[\mathcal{B}, L\mathcal{O}]}$ is an interval in $\Phi$.\\
{\rm (iii)} 
$\Phi_{\overline{[\mathcal{B}, L\mathcal{O}]}}$ consists of all mappings $\varphi$ of $\Theta^1$ into $(\mathcal{L}(\mathcal{CS})\backslash \mathcal{L}(\mathcal{R}e\mathcal{B})) \cup \mathcal{L}(\mathcal{G}) \cup \mathbb{N}_{\,3}^*$ satisfying conditions (P1) to (P4) and the additional conditions:\\

\noindent
(T)* $ \tau\varphi = T^{\ast}$ for some $\tau \in \Theta$.\\

\noindent
(P5)* $(\emptyset\varphi)_{T_{\ell} K^{\ast}} \leq T_{\ell}\varphi,;\; \;
(\emptyset\varphi)_{T_r K^{\ast}} \leq T_r\varphi$,\;\\

\noindent
(P6)* $(T_{\ell}\varphi)_{T_r K^{\ast}} \leq (T_{\ell}T_r)\varphi;\;\;
(T_r\varphi)_{T_{\ell} K^{\ast}} \leq (T_rT_{\ell})\varphi.$\\
{\rm (iv)} $\Phi_{\overline{[\mathcal{B}, L\mathcal{O}]}}$ is a sublattice but not a complete 
sublattice of $\Phi$.
\end{theorem}
{\bf Note.}  Clearly the conditions (P5)* and (P6)* are weaker than conditions (P5) and 
(P6) since we have to deal with arbitrary $\sigma, \tau$ in (P5) and (P6).  The big 
gain is that they are much easier to verify in specific examples. 
\begin{proof}
(i)  Let $\mathcal{P}$ denote the set of all order-preserving mappings of $\Theta^1$ into 
$(\mathcal{L}(\mathcal{CS})\backslash \mathcal{L}(\mathcal{R}e\mathcal{B})) \cup \mathcal{L}(\mathcal{G})$.  
Let $\mathcal{V} \in [\mathcal{B}, L\mathcal{O}]$.  By Lemma 5.1, it follows that $\tau\chi_{\mathcal{V}} 
\supseteq \mathcal{T}$ for all $\tau$.  Now $\tau\chi_{L\mathcal{O}} = \mathcal{CS}$ for all 
$\tau \in \Theta^1$.   By (P2), $\chi$ is order preserving and so, for all $\tau \in \Theta^1$, we have  $\mathcal{T} \subseteq 
\tau\chi_{\mathcal{V}} \subseteq \tau\chi_{L\mathcal{O}} = \mathcal{CS}$ so that 
$\tau\chi_{\mathcal{V}} \in \mathcal{L}(\mathcal{CS})$.  But $\tau\chi_{\mathcal{V}} = 
\mathcal{V}_{\tau K^{\ast}} = 
\mathcal{V}_{\tau K} \in (\mathcal{L}(\mathcal{CS})\backslash \mathcal{L}(\mathcal{R}e\mathcal{B})) \cup \mathcal{L}(\mathcal{G})$.  Since $\chi_{\mathcal{V}} \in \Phi$, it follows that $\chi_{\mathcal{V}}$ 
is order preserving.  Hence $\chi_{\mathcal{V}} \in \Phi_{[\mathcal{B}, L\mathcal{O}]}$\\

\noindent
For the converse, we must first show that $\mathcal{P}\subseteq \Phi$.  
Let $\varphi \in \mathcal{P}$.  Clearly   
$\varphi$ satisfies (P1) and (P2) by the definition of $\mathcal{P}$.  
Conditions (P3) and (P4) are satisfied vacuously since 
$\tau\varphi \in (\mathcal{L}(\mathcal{CS})\backslash \mathcal{L}(\mathcal{R}e\mathcal{B})) 
\cup \mathcal{L}(\mathcal{G})$ for 
all $\tau\in \Theta^1$.  Now let $\tau \in \Theta$.  Then $(\emptyset \varphi)_{\tau} \in 
\{\mathcal{T, LZ, RZ}\}$ 
so that $(\emptyset \varphi)_{\tau K^{\ast}} \in \{T^{\ast}, 
L^{\ast}, R^{\ast} \}$.  Therefore $(\emptyset \varphi)_{\tau K^{\ast}} \leq \mathcal{T} 
\leq \tau\varphi$.  Thus $\varphi$ satisfies (P5).\\

Next let $\sigma ,\tau \in \Theta$ be such that $\sigma
\varphi \in \mathcal{K}_0$ and $t(\sigma ) \neq h(\tau )$.  Then $\sigma \varphi, 
(\sigma\tau)\varphi \in 
(\mathcal{L}(\mathcal{CS})\backslash \mathcal{L}(\mathcal{R}e\mathcal{B})) \cup \mathcal{L}(\mathcal{G})$.  Hence $(\sigma \varphi)_{\tau} \in 
\{\mathcal{T, LZ, RZ}\}$ so that $(\sigma \varphi)_{\tau K^{\ast}} \in \{T^{\ast}, 
L^{\ast}, R^{\ast} \}$.  Therefore $(\sigma \varphi)_{\tau K^{\ast}} \leq \mathcal{T} 
\leq \sigma\tau\varphi$.  Thus $\varphi$ satisfies (P6) so that $\varphi \in \Phi$.  
Consequently $\mathcal{P} \subseteq \Phi.$ \\

\noindent
Now let $\varphi \in \Phi_{[\mathcal{B}, L\mathcal{O}]}$.  Since 
$\Phi_{[\mathcal{B}, L\mathcal{O}]}  \subseteq 
\Phi$, it follows from Theorem 4.5, that there exists a variety $\mathcal{V} \in [\mathcal{S, CR}]$ 
such that $\varphi = \chi_{\mathcal{V}}.$  By Lemma 5.1, $\mathcal{B}\subseteq \mathcal{V}$ so that 
$\mathcal{V}$ actually lies in the interval $[\mathcal{B, CR}]$.   Moreover,  $\mathcal{V}_K = \emptyset\chi_{\mathcal{V}} \in 
(\mathcal{L}(\mathcal{CS})\backslash \mathcal{L}(\mathcal{R}e\mathcal{B})) \cup \mathcal{L}(\mathcal{G}).$   But $\mathcal{CS}^K = L\mathcal{O}$ and so $\mathcal{V} \in \mathcal{L}(L\mathcal{O})$.   Therefore $\mathcal{V} \in 
[\mathcal{B}, L\mathcal{O}]$ as required.\\

\noindent
(ii)  This follows immediately from the fact that $\chi$ is an isomorphism.\\

\noindent
(iii)  Let $\mathcal{P}$ denote the set of $\varphi\in \Phi$ satisfying the conditions layed  
in the statement.  
Let $\mathcal{V} \in \overline{[\mathcal{B}, L\mathcal{O}]}$. 
 By Theorem 4.5, $\chi_{\mathcal{V}} \in \Phi$ 
and therefore $\chi_{\mathcal{V}}$ satisfies conditions (P1) - (P6).    Since (P5)* and (P6)* are 
weaker than (P5) and (P6), respectively, it follows that $\chi_{\mathcal{V}}$ also satisfies 
(P5)* and (P6)*.  By Lemma 5.1, there exists $\tau \in \Theta$ such that  
$\tau\chi_{\mathcal{V}} = T^{\ast}$.  Hence 
$\chi_{\mathcal{V}}$ satisfies (T)* and so $\chi_{\mathcal{V}} \in \mathcal{P}$.\\

\noindent
We must now show that $\mathcal{P}\subseteq \Phi$.  
 Let $\varphi \in \mathcal{P}$.  That   
$\varphi$ satisfies (P1) - (P4) follows from the definition of 
$\Phi_{\overline{[\mathcal{B}, L\mathcal{O}]}}$.    Now let $\tau \in \Theta$.  
We have $(\emptyset \varphi)\in 
(\mathcal{L}(\mathcal{CS})\backslash \mathcal{L}(\mathcal{R}e\mathcal{B})) \cup \mathcal{L}(\mathcal{G})$.  
If $|\tau| \geq 2$ then this implies that  $(\emptyset \varphi)_{\tau} = \mathcal{T}$ and therefore 
$(\emptyset \varphi)_{\tau K^{\ast}} = T^{\ast} \leq \tau\varphi$.  On the other hand, if 
$|\tau| = 1$, that is $\tau \in \{T_{\ell}, T_r\}$, then $(\emptyset\varphi)_{\tau} \leq \tau\varphi$ by 
(P5)*.  Hence, $\varphi$ satisfies (P5).  \\

Next let $\sigma ,\tau \in \Theta$ be such that $\sigma
\varphi \in \mathcal{K}_0$ and $t(\sigma ) \neq h(\tau )$.  Then $\sigma \varphi  \in 
(\mathcal{L}(\mathcal{CS})\backslash \mathcal{L}(\mathcal{R}e\mathcal{B})) \cup \mathcal{L}(\mathcal{G})$.  Hence, if $|\tau| \geq 2$ then  $(\sigma \varphi)_{\tau} = \mathcal{T}$ and $(\sigma \varphi)_{\tau K^{\ast}} 
= T^{\ast}$.  Hence $(\sigma \varphi)_{\tau K^{\ast}} = T^{\ast} \leq (\sigma\tau)\varphi$.  On the 
other hand, if $|\tau| = 1$, that is $\tau \in \{T_{\ell}, T_r\}$, then 
$(\sigma \varphi)_{\tau K^{\ast}} \leq (\sigma\tau)\varphi$ by (P6)*.  Thus $\varphi$ satisfies (P6) 
and  $\mathcal{P} \subseteq \Phi$.\\
Conversely, let $\varphi \in \mathcal{P}$.  Since 
$\Phi_{\overline{[\mathcal{B, O}]}} \subseteq 
\Phi$, it follows from Theorem 4.5, that there exists a variety $\mathcal{V} \in [\mathcal{S, CR}]$ 
such that $\varphi = \chi_{\mathcal{V}}.$  By Lemma 5.1, $\mathcal{B} \nsubseteq \mathcal{V}$.  
We also have $\mathcal{V}_K = \emptyset\chi_{\mathcal{V}} \in (\mathcal{L}(\mathcal{CS})\backslash \mathcal{L}(\mathcal{R}e\mathcal{B})) \cup \mathcal{L}(\mathcal{G}).$   But $\mathcal{CS}^K = L\mathcal{O}$ and so $\mathcal{V} \in \mathcal{L}(L\mathcal{O})$.  
Therefore $\mathcal{V} \in \overline{[\mathcal{B}, L\mathcal{O}]}$ as required.\\

\noindent
(iv)  
Now let $\varphi, \theta \in \Phi_{\overline{[\mathcal{B}, L\mathcal{O}]}}$.  Let $\tau \in \Theta^1$.  
Then both $\tau\varphi$ and $\tau\theta$ lie in $\mathcal{K}_{\mathcal{CS}}$ and therefore 
so also does $\tau(\varphi \vee \theta) = \tau\varphi \vee \tau\theta$.  Also, it is clear 
that $\varphi \vee \theta$  
satisfies condition (T)*.   Now $\varphi \vee \theta \in \Phi$.  Hence $\varphi \vee \theta$ 
satisfies (P1) - (P6) and therefore also (P1) - (P4) and (P5)*, (P6)*.  In other words, 
$\varphi \vee \theta \in \Phi_{\overline{[\mathcal{B}, L\mathcal{O}]}}$.   An almost identical 
argument, with slight adjustment, will now establish that 
$\varphi \wedge \theta \in \Phi_{\overline{[\mathcal{B}, L\mathcal{O}]}}$.  Thus 
$\Phi_{\overline{[\mathcal{B}, L\mathcal{O}]}}$ is a sublattice of $\Phi$.  The same argument 
as in Theorem 6.1(iv) will show that $\Phi_{\overline{[\mathcal{B}, L\mathcal{O}]}}$ is not a 
complete sublattice of $\Phi$.
\end{proof}

There is a great range in the complexity of $K$-classes, even  within $\mathcal{L}(\mathcal{O})$.   
Of course, the best known $K$-class of all is $\mathcal{L}(\mathcal{B}) = \mathcal{T}K$.  Not surprisingly, 
the picture becomes more complicated as soon as we look outside  $\mathcal{L}(\mathcal{B})$. 
But there is much that can be said ranging from complete intervals that can be determined to 
extremely complicated $K$-classes.  We conclude with several observations 
illustrating the complexity of $K$-classes in general, even at a low level within 
$\mathcal{L}(\mathcal{CR})$.  Both $K$ and $K_{\ell}$ are complete congruences.  Since 
$K_{\ell}$ is a finer congruence than $K$, we can gain some insight into the complexity of an 
individual $K$-class by considering the quotient relation $K/K_{\ell}$.  We show that 
even at the level of $\mathcal{LRO}$ we have an interval within a $K_{\ell}$ that is isomorphic 
to $\mathcal{L}(\mathcal{G})$.  Recall that for $\mathcal{P} 
\in \mathcal{L}(\mathcal{G})$, we have 
$\mathcal{P}K = [\mathcal{P}, \mathcal{O}H\mathcal{P}]$ and  that for any $\mathcal{Q} \in 
\mathcal{L}(\mathcal{CS})\backslash\mathcal{L}(\mathcal{R}e\mathcal{G})$, we have 
$\mathcal{Q}K = [\mathcal{Q}, 
L(\mathcal{O}D\mathcal{Q})]$.\\

\begin{theorem}  Let $\mathcal{P} \in \mathcal{L}(\mathcal{G}), \mathcal{Q}\in 
\mathcal{CS}\backslash(\mathcal{R}e\mathcal{G})$.\\
{\rm (i)} $\mathcal{P}K/K_{\ell}$
 contains a sublattice isomorphic to $\mathcal{L}(\mathcal{P})$.\\
{\rm (ii)} $\mathcal{Q}K/K_{\ell}$ 
contains a sublattice isomorphic to 
$(\mathcal{L}(\mathcal{CS})\backslash \mathcal{L}(\mathcal{R}e\mathcal{G})) \cup 
\mathcal{L}(\mathcal{G})$.\\
{\rm (iii)}  $\mathcal{G}K/K_{\ell}$ and $\mathcal{CS}/K_{\ell}$
both have the cardinality of the continuum.  \\
{\rm (iv)} \begin{eqnarray*}
\mathcal{LRO}K_{\ell} &=& [\mathcal{S}\vee\mathcal{G}, \mathcal{LRO}]\\
\mathcal{LRO}(K_{\ell}\cap {\mathbf B}) &=& [\mathcal{LRB}\vee \mathcal{G}, \mathcal{LRO}]\\
&\cong& \mathcal{L}(\mathcal{G}).
\end{eqnarray*}
\end{theorem}
\begin{proof}
(i) Let $\mathcal{U}\in \mathcal{L}(\mathcal{P})$ and define the mapping
 $\varphi_{\mathcal{U}}: \Theta^1 \longrightarrow \mathcal{L}(\mathcal{P})$ 
as follows:
\[ \tau\varphi_{\mathcal{U}} = \left\{ \begin{array}{ll}
\mathcal{P}  &  \mbox{if $\tau = \emptyset$}\\
\mathcal{U}  &  \mbox{otherwise}.
\end{array}
\right. \]
Clearly $\varphi_{\mathcal{U}}$ is order preserving and therefore, by Theorem 6.1, 
belongs to $\Phi_{[\mathcal{B, O}]}$.  It is also evident, by examining the corresponding 
ladders, that the mapping $\varphi:
\mathcal{U} \longrightarrow \varphi_{_{\mathcal{U}}}$ is a monomorphism of 
$\mathcal{L}(\mathcal{P})$ into $\Phi$.

By Theorem 4.5 and 6.1, $\varphi_{\mathcal{U}} = \chi_{_{\mathcal{V}}}$, for 
$\mathcal{V} = \varphi_{\mathcal{U}}\xi \in 
[\mathcal{B, O}]$.  Then $\mathcal{V}_K = \emptyset \chi_{_{\mathcal{V}}} = 
\emptyset\varphi_{\mathcal{U}} = \mathcal{P}$, so that $\mathcal{V}\in \mathcal{P}K$.  It is now 
routine to check that $\{\varphi_{\mathcal{U}}|\mathcal{U}\in \mathcal{L}(\mathcal{P})\}$ is a 
sublattice of $\Phi$ and therefore that 
$\{\varphi_{\mathcal{U}}\xi|\mathcal{U}\in \mathcal{L}(\mathcal{P})\}$ is a sublattice of 
$\mathcal{P}K$ which is isomorphic to $\mathcal{L}(\mathcal{P})$ via the isomorphism 
$\mathcal{U} \longrightarrow \varphi_{\mathcal{U}}\xi$.  By Lemma 4.7 and the remarks following that 
lemma, we see that, for distinct varieties $\mathcal{U}_1, \mathcal{U}_2 \in \mathcal{L}(\mathcal{G})$, 
the varieties of the form $\varphi_{\mathcal{U}_i}\xi, i =1,2$ lie in distinct $T_{\ell}$-classes.  But 
such varieties lie in the same $K$-class $\mathcal{P}K$.  Hence these varieties lie in distinct 
$K_{\ell}$-classes in $\mathcal{P}K$ and the mapping 
$\mathcal{U} \longrightarrow \varphi_{\mathcal{U}}\xi K_{\ell}$ is 
a lattice embedding of $\mathcal{L}(\mathcal{P})$ into $(\mathcal{P}K)/K_{\ell}$.\\

\noindent
(ii)  Let $\mathcal{U} \in (\mathcal{L}(\mathcal{Q})\backslash \mathcal{L}(\mathcal{R}e\mathcal{G})) 
\cup \mathcal{L}(\mathcal{Q}\cap\mathcal{G})$.  Note that $(\mathcal{L}(\mathcal{Q})\backslash 
\mathcal{L}(\mathcal{R}e\mathcal{G})) 
\cup \mathcal{L}(\mathcal{Q}\cap\mathcal{G})$ is not a sublattice of $\mathcal{L}(\mathcal{CR})$ but 
it is a lattice in the inherited order.  Define the mapping
 $$\varphi_{\mathcal{U}}: \Theta^1 \longrightarrow 
(\mathcal{L}(\mathcal{CS})\backslash \mathcal{L}(\mathcal{R}e\mathcal{B})) \cup 
\mathcal{L}(\mathcal{G})$$ 
as follows:
\[ \tau\varphi_{\mathcal{U}} = \left\{ \begin{array}{ll}
\mathcal{Q}  &  \mbox{if $\tau = \emptyset$}\\
\mathcal{U}  &  \mbox{otherwise}.
\end{array}
\right. \]
Clearly $\varphi_{\mathcal{U}}$ is order preserving and therefore belongs to $\Phi_{[\mathcal{B}, 
L\mathcal{O}]}$.  By Theorems 4.5 and 6.2, 
$\varphi_{\mathcal{U}} = \chi_{_{\mathcal{V}}}$, for some $\mathcal{V} 
= \varphi_{\mathcal{U}}\xi \in 
[\mathcal{B}, L\mathcal{O}]$.  Then $\mathcal{V}_K = \emptyset \chi_{_{\mathcal{V}}} = 
\emptyset\varphi_{\mathcal{U}} = \mathcal{Q}$, so that $\mathcal{V}\in \mathcal{Q}K$.  It is now 
routine to check that $\{\varphi_{\mathcal{U}}\xi|\mathcal{U}\in 
(\mathcal{L}(\mathcal{Q})\backslash \mathcal{L}(\mathcal{R}e\mathcal{B})) \cup 
\mathcal{L}(\mathcal{Q} \cap \mathcal{G})
\}$ is a 
sublattice of $\mathcal{Q}K$ which is isomorphic to $
(\mathcal{L}(\mathcal{Q})\backslash \mathcal{L}(\mathcal{R}e\mathcal{B})) \cup 
\mathcal{L}(\mathcal{Q}\cap \mathcal{G})$ via the isomorphism 
$\mathcal{U} \longrightarrow \varphi_{\mathcal{U}}\xi$.
By Lemma 4.7 and the remarks following that 
lemma, we see that, for distinct varieties $\mathcal{U}_1, \mathcal{U}_2 \in 
(\mathcal{L}(\mathcal{Q})\backslash \mathcal{L}(\mathcal{R}e\mathcal{B})) \cup 
\mathcal{L}(\mathcal{Q}\cap \mathcal{G})$,
the varieties of the form $\varphi_{\mathcal{U}_i}\xi, i =1,2$ lie in distinct $T_{\ell}$-classes.  But 
such varieties lie in the same $K$-class $\mathcal{Q}K$.  Hence these varieties lie in distinct 
$K_{\ell}$-classes in $\mathcal{Q}K$ and the mapping 
$\mathcal{U} \longrightarrow \varphi_{\mathcal{U}}\xi K_{\ell}$ is 
an embedding of 
$(\mathcal{L}(\mathcal{Q})\backslash \mathcal{L}(\mathcal{R}e\mathcal{B})) \cup 
\mathcal{L}(\mathcal{Q}\cap \mathcal{G})$
 into $(\mathcal{Q}K)/K_{\ell}$. \\

\noindent
(iii)  This follows from the fact that $\mathcal{L}(\mathcal{G})$ has the cardinality of the 
continuum, see  [Ad].\\

\noindent
(iv)  Recall from Lemma 2.1(iii) that $\mathcal{LRO}T_{\ell} = [\mathcal{S}, \mathcal{LRO}]$ 
while clearly $\mathcal{LRO} \in \mathcal{G}K$.  Hence

\begin{eqnarray*}
\mathcal{LRO}K_{\ell} &=& \mathcal{LRO}K\cap \mathcal{LRO}T_{\ell} = [\mathcal{G}, \mathcal{O}] 
\cap [\mathcal{S, LRO}]\\
&=& [\mathcal{S}\vee \mathcal{G}, \mathcal{LRO}]
\end{eqnarray*}  
which establishes the first equality.  Next note that $\mathcal{LRO}\cap \mathcal{B} = 
\mathcal{LRB}$ so that $\mathcal{LRO}\subseteq \mathcal{LRB}^{\mathcal{B}}$.  But also 
we have $\mathcal{LRB}^{\mathcal{B}} \cap \mathcal{B} = \mathcal{LRB}$.  Hence 
$\mathcal{LRB}^{\mathcal{B}} \subseteq \mathcal{LRO}$ and equality prevails: 
$\mathcal{LRB}^{\mathcal{B}} = \mathcal{LRO}$.  We then have
\begin{eqnarray*}
\mathcal{LRO}(K_{\ell}\cap {\mathbf{B}}) &=& \mathcal{LRO}K_{\ell}\cap \mathcal{LRO}{\mathbf{B}}\\
&=& [\mathcal{S}\vee \mathcal{G}, \mathcal{LRO}] \cap [\mathcal{LRB}, \mathcal{LRB}^{\mathcal{B}}]\\
&=&  [\mathcal{S}\vee \mathcal{G}, \mathcal{LRO}] \cap [\mathcal{LRB}, \mathcal{LRO}]\\
&=& [\mathcal{LRB} \vee \mathcal{G}, \mathcal{LRO}].
\end{eqnarray*}
That establishes the second equality in part (iv) and we now consider the final isomorphism.  
Let $\mathcal{U}\in \mathcal{L}(\mathcal{G})$ and define the mapping
 $\varphi_{\mathcal{U}}: \Theta^1 \longrightarrow \mathcal{L}(\mathcal{P})$ 
as follows:
\[ \tau\varphi_{\mathcal{U}} = \left\{ \begin{array}{ll}
\mathcal{G}  &  \mbox{if $\tau = \emptyset$}\\
\mathcal{U}  &  \mbox{if $\tau = T_r$}\\
T^{\ast}  &  \mbox{otherwise}.
\end{array}
\right. \]

\noindent
Clearly, $\varphi_{_{\mathcal{U}}}$ satisfies the conditions in Theorem 6.1(iii).   Therefore, 
in particular, $\varphi{_{_\mathcal{U}}} \in \Phi$.  Let $\varphi$ denote the mapping 
$\mathcal{U} 
\longrightarrow \varphi_{_{\mathcal{U}}}$ of $\mathcal{L}(\mathcal{G})$ into $\Phi$.  Clearly 
$\varphi$ is a monomorphism.    It is also evident that $\mathcal{L}(\mathcal{G})\varphi$ 
is an interval in $\Phi$ and therefore that $\mathcal{L}(\mathcal{G})\varphi\xi$ is 
an interval in $\mathcal{L}(\mathcal{CR})$.  However, the least and greatest elements 
in  $\mathcal{L}(\mathcal{G})\varphi$ are $\varphi_{_{\mathcal{T}}}$ and 
$\varphi_{_{\mathcal{G}}}$, respectively.   But, on inspection, it is clear that 
$\varphi_{_{\mathcal{T}}} = \chi_{_{\mathcal{LRB}\vee \mathcal{G}}}$ and 
$\varphi_{_{\mathcal{G}}} = 
\chi_{_{\mathcal{LRO}}}$.  Consequently, we see that the mapping $\varphi\xi$ is an 
isomorphism of $\mathcal{L}(\mathcal{G})$ to the interval 
 $[\mathcal{LRB}\vee \mathcal{G}, \mathcal{LRO}]$, as required.
\end{proof}

This concludes our discussion of the general background required for further more focussed 
studies of particular $K$-classes to be presented in Part II.


\vskip1.6cm

\noindent
{\small
Department of Mathematics \\
Simon Fraser University \\
Burnaby, British Columbia\\
Canada V5A 1S6 \\[0.2cm]
Email: nreilly@sfu.ca
}
\end{document}